\documentclass[11pt]{amsart}

\usepackage{amsmath,amssymb,latexsym,soul,cite,mathrsfs}

\usepackage{color,enumitem,graphicx}
\usepackage[colorlinks=true,urlcolor=blue,
citecolor=red,linkcolor=blue,linktocpage,pdfpagelabels,
bookmarksnumbered,bookmarksopen]{hyperref}
\usepackage[english]{babel}

\usepackage[left=2.9cm,right=2.9cm,top=2.8cm,bottom=2.8cm]{geometry}
\usepackage[hyperpageref]{backref}

\usepackage[colorinlistoftodos]{todonotes}
\makeatletter
\providecommand\@dotsep{5}
\def\listtodoname{List of Todos}
\def\listoftodos{\@starttoc{tdo}\listtodoname}
\makeatother

\numberwithin{equation}{section}

\newtheorem{theorem}{Theorem}[section]
\newtheorem{proposition}[theorem]{Proposition}
\newtheorem{lemma}[theorem]{Lemma}
\newtheorem{corollary}[theorem]{Corollary}

\newtheorem{remark}{Remark}

\newcommand\R{\mathbb R}

\begin{document}
	
	\title [Uniqueness of Heteroclinic Solutions in a Class of Autonomous Quasilinear ODE Problems]{Uniqueness of Heteroclinic Solutions in a Class of Autonomous Quasilinear ODE Problems}

	\author{Claudianor O. Alves}
	\author{Renan J. S. Isneri}
	\author{Piero Montecchiari}

	\address[Claudianor O. Alves]{\newline\indent Unidade Acad\^emica de Matem\'atica
		\newline\indent 
		Universidade Federal de Campina Grande,
		\newline\indent
		58429-970, Campina Grande - PB - Brazil}
	\email{\href{mailto:coalves@mat.ufcg.edu.br}{coalves@mat.ufcg.edu.br}}
	
	\address[Renan J. S. Isneri]
	{\newline\indent Unidade Acad\^emica de Matem\'atica
		\newline\indent 
		Universidade Federal de Campina Grande,
		\newline\indent
		58429-970, Campina Grande - PB - Brazil}
	\email{\href{renan.isneri@academico.ufpb.br}{renan.isneri@academico.ufpb.br}}
	
	\address[Piero Montecchiari]
	{\newline\indent Dipartimento di Ingegneria Civile, Edile e Architettura,
		\newline\indent 
		Universit\`a Politecnica delle Marche,
		\newline\indent
		Via Brecce Bianche, I–60131 Ancona, Italy}
	\email{\href{p.montecchiari@staff.univpm.it}{p.montecchiari@staff.univpm.it}}

	\pretolerance10000

	\begin{abstract}
		\noindent In this paper, we prove the existence, uniqueness and qualitative properties of heteroclinic solution for a class of autonomous quasilinear ordinary differential equations of the Allen-Cahn type given by
		$$
		-\left(\phi(|u'|)u'\right)'+V'(u)=0~~\text{ in }~~\mathbb{R},
		$$
		where  $V$ is a double-well potential with minima at $t=\pm\alpha$ and $\phi:(0,+\infty)\to(0,+\infty)$ is a $C^1$ function satisfying some technical assumptions. Our results include the classic case $\phi(t)=t^{p-2}$, which is related to the celebrated $p$-Laplacian operator, presenting the explicit solution in this specific scenario. Moreover, we also study the case $\phi(t)=\frac{1}{\sqrt{1+t^2}}$, which is directly associated with the prescribed mean curvature operator.
	\end{abstract}

	
	\maketitle
	
	\section{Introduction}

	In a recent paper \cite{Alves1}, (see also \cite{Alves3}), we studied the problem of existence and multiplicity of saddle-type solutions for the class of quasilinear elliptic equation on $\mathbb{R}^{2}$
	$$
	-\Delta_{\Phi}u + V'(u)=0\eqno{(PDE)}
	$$ 
	where $\Delta_{\Phi}u=\text{div}(\phi(|\nabla u|)\nabla u)$ 
	and $\Phi:\mathbb{R}\rightarrow [0,+\infty)$ is a N-function of the form
	$\Phi(t)=\int_0^{|t|}s\phi(s)ds$ for a $C^{1}$ function $\phi:[0,+\infty)\rightarrow [0,+\infty)$ satisfying suitable assumptions while $V$ is an even double well potential (see $(\phi_{1})$--$(\phi_{4})$ and $(V_{1})$--$(V_{5})$ below).  The aim of that study was to generalize to the quasilinear cases earlier results on saddle solutions on $\R^{2}$ (see \cite{Alessio0, Alessio1, Alessio2, Alessio3}) obtained for  elliptic equations 
	\begin{equation}\label{001}
		-\Delta u+V'(u)=0
	\end{equation}
	when V is an even double well potential patterned  on the classical Ginzburg-Landau potential
	\begin{equation}\label{002}
		V(u)=\frac{1}{4}(u^2-1)^2.
	\end{equation}
	In \cite{Alves1} it is shown that, for any $k \in \mathbb{N}$, (PDE) has a planar solution with a nodal set that consists of the union of $k$ lines passing through the origin, forming angles equal to $\pi/k$. 
	Along directions parallel to the nodal lines, that solution is asymptotic to a heteroclinic solution (suitably rotated and reflected) of the related  one dimensional quasilinear ODE
	\begin{equation}\label{onedimeeq}
		-\left(\phi(|u'|)u'\right)'+V'(u)=0.
	\end{equation}
	In different directions, the solution approaches the different minima of $V$ alternately with respect to the nodal lines.
	
	To gain a deeper insight into saddle solutions, this paper aims to thoroughly investigate the qualitative properties of the heteroclinic solutions in \eqref{onedimeeq}. As described below, our analysis shows in particular that equation \eqref{onedimeeq} has exactly one heteroclinic solution up to translation. This result is well known for the classical case  and is in some way related to celebrated ``rigidity'' results concerning entire solutions of the Allen Cahn problem \eqref{001} such as the De Giorgi's conjecture or the Gibbon's conjecture. In contrast to the De Giorgi conjecture, which remains an open problem in its full generality (see partial results in \cite{Gui, Ambrosio,[Sa]}), the Gibbon's conjecture has been proved with different methods by Farina \cite{Farina}, Barlow, Bass and Gui \cite{Barlow}, and Berestycki, Hamel, and Monneau \cite{berestycki}. These results assure that for a broad class of potential $V$ modeled on \eqref{002} and for any $N>1$, if $u$ is a bounded solution on $\mathbb{R}^{N}$ of \eqref{001} such that
	$\lim_{x_N\to\pm\infty}u(x',x_N)=\pm1$ uniformly with respect to $x'\in\mathbb{R}^{N-1}$ then
	$u(x)=\theta(x_{N}-b)$ for some $b\in\mathbb{R}$, for all $x\in\mathbb{R}^N$,  where $\theta$ is the unique solution of the ODE problem
	\[-\ddot q+V'(q)=0,\quad q(0)=0,\quad q(\pm\infty)=\pm1.\]
	This type of result has been extended to other differential operators, such as the case of $p$-Laplacian operators, as studied in \cite{Esposito}, or the case of fractional operators (see the survey \cite{Chan} and the references therein). It would be interesting to investigate whether the quasilinear equation (PDE) behaves similarly and our study can be considered in a sense a preliminary step in this direction.\medskip
	
	To be more precise, in the present paper we first investigate the existence and uniqueness of solution for the class of quasilinear problems
\begin{equation}\label{007}
	-\left(\phi(|u'|)u'\right)'+V'(u)=0~\text{ in }~\mathbb{R},~u(0)=0,~u(\pm\infty)=\pm\alpha,
\end{equation}
where $V:\mathbb{R}\to\mathbb{R}$ is a double-well potential with absolute minima at $t=\pm\alpha$ and $\phi:(0,+\infty)\rightarrow [0,+\infty)$ is a $C^1$ function verifying the following conditions
\begin{itemize}
	\item[($\phi_1$)] $\phi(t)>0$ and $\left(\phi(t)t\right)'>0$ for any $t>0$.
	\item[($\phi_2$)] There are $l,m\in\mathbb{R}$ with $1<l\leq m$ such that 
	$$
	l-1\leq \dfrac{(\phi(t)t)'}{\phi(t)}\leq m-1,~~\forall t>0.
	$$
	\item[($\phi_3$)] There exist constants $c_1,c_2,\delta>0$ and $q>1$ satisfying
	$$
	c_1t^{q-1}\leq \phi(t)t\leq c_2t^{q-1},~~t\in(0,\delta].
	$$	
	\item[($\phi_4$)] $\phi:[0,+\infty)\rightarrow [0,+\infty)$ and $\phi$ is non-decreasing.
\end{itemize}

The hypotheses $(\phi_1)$ and $(\phi_2)$ are well known and they guarantee that the function $\Phi:\mathbb{R}\rightarrow [0,+\infty)$ given by
\begin{equation}\label{*}
	\Phi(t)=\int_0^{|t|}s\phi(s)ds
\end{equation} 
is an $N$-function that checks the so called $\Delta_2$-condition (see for instance Section \ref{s2}). These conditions ensures that $\Phi$ does not increase more rapidly than exponential functions, and this fact yields that the Orlicz-Sobolev spaces associated with $\Phi$ are reflexive and separable.

In this work, the potential $V$ in general satisfies the following conditions:

\begin{itemize}
	\item[$(V_1)$] $V\in C^1(\mathbb{R},\mathbb{R})$ and $V$ is a non-negative function on $\mathbb{R}$.
	\item[$(V_2)$] There is $\alpha>0$ such that $V(\alpha)=V(-\alpha)=0$ and $V(t)>0$ for all $t\in(-\alpha,\alpha)$.
	\item[$(V_3)$] There are $a_1,a_2,a_3,a_4>0$ and $\delta\in (0,\alpha)$ such that
	$$
	a_1\Phi(|t-\alpha|)\leq V(t)\leq a_2\Phi(|t-\alpha|),~\forall t\in(\alpha-\delta,\alpha]
	$$
	and 
	$$
	a_3\Phi(|t+\alpha|)\leq V(t)\leq a_4\Phi(|t+\alpha|),~\forall t\in[-\alpha,-\alpha+\delta).
	$$
	\item[$(V_4)$] There exist positive real numbers $c_i$ with $i=1,...,8$ such that
	$$
	-c_3t\phi(c_4|t-\alpha|)|t-\alpha|\leq V'(t)\leq-c_1t\phi(c_2|t-\alpha|)|t-\alpha|\text{ for all }t\in[0,\alpha]
	$$
	and 
	$$
	-c_7t\phi(c_8|t+\alpha|)|t+\alpha|\leq V'(t)\leq -c_5t\phi(c_6|t+\alpha|)|t+\alpha|\text{ for all }t\in[-\alpha,0].
	$$
	\item[$(V_5)$] There exists $\rho>0$ such that $V$ is strictly convex on $B_\rho(-\alpha)=(-\alpha-\rho,-\alpha+\rho)$ and on $B_{\rho}(+\alpha)=(\alpha-\rho,\alpha+\rho)$.
\end{itemize}

As an example, given a $\phi$ satisfying $(\phi_1)$-$(\phi_2)$ and letting $\Phi$ be defined by \eqref{*} , the function 
\begin{equation}\label{009}
	V(t)=\Phi(|t^2-\alpha^2|),~~t\in\mathbb{R},
\end{equation}
satisfies $(V_1)$-$(V_5)$. This type of potential $V$  has been widely used and investigated in \cite{Alves1,Alves2,Alves3}. 

Now, we are in a position to state our first result as follows.

\begin{theorem}\label{T1}
Assume $(\phi_1)$-$(\phi_3)$, $(V_1)$-$(V_5)$. Then the problem \eqref{007} has a unique twice differentiable solution $q$ in $C^{1,\gamma}_{\text{loc}}(\mathbb{R})$ for some $\gamma\in(0,1)$. Moreover,
	\begin{itemize}
		\item[(a)] If $(\phi_4)$ occurs then there are $\theta_i,\beta_i>0$ such that
			\begin{equation}\label{exp1}
				0<\alpha-q(t)\leq\theta_1e^{-\theta_2t}~~\text{ and }~~0<q'(t)\leq \beta_1e^{-\beta_2t}~~\text{ for all }t\geq0
			\end{equation}
			and 
			\begin{equation}\label{exp2}
				0<\alpha+q(t)\leq \theta_3e^{\theta_4t}~~\text{ and }~~0<q'(t)\leq \beta_3e^{\beta_4t}~~\text{ for all }t\leq0.
			\end{equation}
		\item[(b)] If $V$ is the potential given in \eqref{009}, then $q$ satisfies the following estimates
		\begin{itemize}
			\item[(i)] if $l-1\leq1\leq m-1$, then 
			$$
			\alpha\tanh\left(\frac{\alpha t}{m-1}\right)\leq q(t)\leq \alpha\tanh\left(\frac{\alpha t}{l-1}\right)~\text{ for }t\geq0,
			$$ 
			\item[(ii)] if $l-1>1$, then 
			$$
			\alpha\tanh\left(\frac{\alpha  t}{m-1}\right)\leq q(t)\leq \alpha\tanh\left(\alpha t\right)~\text{ for }t\geq0,
			$$
			\item[(iii)] if $m-1<1$, then 
			$$
			\alpha\tanh\left(\alpha t\right)\leq q(t)\leq\alpha\tanh\left(\frac{\alpha  t}{l-1}\right)~\text{ for }t\geq0. $$
		\end{itemize}
	\end{itemize}
\end{theorem}

By way of illustration, we would like to point out that classic examples of $N$-function $\Phi$ that fit the assumptions of Theorem \ref{T1} are:
\begin{itemize}
	\item[(1)] $\Phi_1(t)=\dfrac{|t|^p}{p}$ with $p\in(1,+\infty)$,
	\item[(2)] $\Phi_2(t)=\dfrac{|t|^p}{p}+\dfrac{|t|^q}{q}$ with $p<q$ and $p,q\in(1,+\infty)$,
	\item[(3)] $\Phi_3(t)=(1+t^2)^{\gamma}-1$ with $\gamma>1$,
	\item[(4)] $\Phi_4(t)=|t|^p\ln(1+|t|)$ for $p\in(1,+\infty)$,
	\item[(5)] $\Phi_5(t)=\displaystyle\int_{0}^{t}s^{1-\gamma}(\sinh^{-1}s)^{\beta}ds$ with $0\leq\gamma< 1$ and $\beta>0$,
\end{itemize}
which appear in many physics applications involving quasilinear equations.

In our study, the proof of the existence of a solution for the problem \eqref{007} was given by a minimization argument to get a minima of the action functional
$$
F(u)=\int_{\mathbb{R}}\left(\Phi(|u'|)+V(u)\right)dt
$$
on the following class of admissible functions
$$
\Gamma_\Phi(\alpha)=\left\{u\in W^{1,\Phi}_{\text{loc}}(\mathbb{R})\mid~u(\pm\infty)=\pm\alpha\right\},
$$
where $W^{1,\Phi}_{\text{loc}}(\mathbb{R})$ denotes the usual Orlicz-Sobolev space. On the other hand, to obtain the uniqueness of the solution we use a local minimization argument that exploits the strict convexity of $V$ around $\pm\alpha$. Moreover, to achieve the estimates in item $(b)$ of Theorem \ref{T1}, we introduce some new ideas, such as reducing the study of \eqref{007} to the following Cauchy problem
 \begin{equation}\label{C}
 	\left\{\begin{array}{ll}
 		y'=G^{-1}(V(y))&\mbox{on}\quad \mathbb{R},
 		\\
 		y(0)=0,
 	\end{array}\right.
 \end{equation}
where
$$
G(t)=\int_0^ts\left(\phi'(s)s+\phi(s)\right)ds,~~~t\geq0.
$$
This Cauchy problem is highly versatile and becomes a useful tool in the investigation of various qualitative properties of the problem \eqref{007} as a whole, as well as in the explicit determination of the solution in some specific cases.

When $V$ is the potential given in \eqref{009}, the exponential decay estimates given in \eqref{exp1} and \eqref{exp2} for the solution $q$ of \eqref{007} and its derivative $q'$ are refined to items $(i)$ to $(iii)$ of Theorem \ref{T1}, making evident that $q$ has geometric properties similar to the function $t\mapsto \alpha\tanh(t)$. However, in particular cases, this is exactly what happens as stated in the following. 
 
\begin{theorem}\label{T2}
	Assume $p\in(1,+\infty)$, 
	$$
	\phi(t)=t^{p-2}~~\text{ and }~~V(t)=\frac{|t^2-\alpha^2|^p}{p}.
	$$
	Then the unique solution $q$ of problem \eqref{007} is given by 
	$$
	q(t)=\alpha\tanh\left(\frac{\alpha t}{\sqrt[p]{p-1}}\right).
	$$
\end{theorem} 

Motivated by recent works \cite{Alves2,Alves4} about heteroclinic solutions for some classes of prescribed mean curvature equations, we also intend to study the existence, uniqueness and qualitative properties of solution for the following problem
\begin{equation}\label{010}
	-\left(\frac{u'}{\sqrt{1+(u')^2}}\right)'+V'(u)=0~\text{ in }~\mathbb{R},~u(0)=0,~u(\pm\infty)=\pm\alpha.
\end{equation}

We could write the above problem in the form \eqref{007} with
\begin{equation}\label{019}
	 \phi(t)=\frac{1}{\sqrt{1+t^2}}.
\end{equation} 
However, in this case $\phi$ does not satisfy condition $(\phi_2)$, making it unfeasible to use the study done for that case. The idea here is to reduce the study of \eqref{010} to a problem of type \eqref{007}, and for that, it was necessary to truncate the function $\phi$ in \eqref{019} and develop new estimates. To present the study of \eqref{010}, given $\alpha>0$ we will assume that the potential $V$ is in $C^2(\mathbb{R},\mathbb{R})$ and that it satisfies $(V_1)$-$(V_2)$ and the properties below: 
\begin{itemize}
	\item[$(V_6)$] $V''(-\alpha),V''(\alpha)>0$.
	\item[$(V_7)$] There are $d_1,d_2,4_3,d_4>0$ such that 
	$$-d_1t|t-\alpha|\leq V'(t)\leq -d_2t|t-\alpha| \text{ for all }t\in[0,\alpha]$$ and $$-d_3t|t+\alpha|\leq V'(t)\leq -d_4t|t+\alpha| \text{ for all }t\in[-\alpha,0].$$
\end{itemize} 
Moreover, we will still assume that
\begin{itemize}
	\item[$(V_8)$] There are $\tilde{\alpha}>0$ and $M(\tilde{\alpha})>0$ such that $\displaystyle\sup_{|t|\in[0,\alpha]}|V'(t)|\leq M(\tilde{\alpha})$ for all $\alpha\in(0,\tilde{\alpha})$. 
\end{itemize} 

The assumption $(V_8)$ is a uniform condition on the potentials $V$ that depends on $\alpha>0$ and a class of such potentials for which $(V_1)$-$(V_2)$ and $(V_6)$-$(V_8)$ are all satisfied is
$$
V(t)=|t^2-\alpha^2|^2,~~\alpha>0.
$$
We can now state the following result

\begin{theorem}\label{T3}
	Assume $V\in C^2(\mathbb{R},\mathbb{R})$, $(V_1)$-$(V_2)$, $(V_6)$-$(V_8)$. Then, there exists  $\alpha_0>0$ such that for each $\alpha\in(0,\alpha_0)$ the problem \eqref{010} has a unique twice differentiable solution $q$ in $C^{1,\gamma}_{\text{loc}}(\mathbb{R})$ for some $\gamma\in(0,1)$ satisfying the exponential decay estimates like those given in \eqref{exp1} and \eqref{exp2}.
	Moreover, if $V(t)=\left(t^2-\alpha^2\right)^2$ then $q$ satisfies 
	\begin{equation}\label{th0}
			\alpha\tanh\left(\alpha\sqrt{2} t\right)\leq q(t)\leq \alpha\tanh\left(\frac{\alpha\sqrt{2}t}{\kappa}\right)~\text{ for }~t\geq 0,
	\end{equation}
	where $\kappa$ is a positive constant that depends on $\|q'\|_{L^{\infty}(\mathbb{R})}$.
\end{theorem}

We would like to point out here that the assumption that the global minima of $V$ are symmetric is made only for notational convenience. Therefore, the entire text can be replicated for the case where the potential $V$ has absolute minima at $\alpha$ and $\beta$ with $\alpha<\beta$ (For more details, see Section \ref{s6}). In these configurations, we would like to take the opportunity to write the following special case

\begin{theorem}\label{T4}
	Assume that $p\in(1,+\infty)$, $\alpha<\beta$ and $V(t)=\frac{|(t-\alpha)(t-\beta)|^p}{p}.$ Then the unique solution $q$ of problem 
	\begin{equation*}
		-\left(|u'|^{p-2}u'\right)'+V'(u)=0~\text{ in }~\mathbb{R},~u(0)=\frac{\alpha+\beta}{2},~u(-\infty)=\alpha,~u(+\infty)=\beta,
	\end{equation*}
	is given by 
	$$
	q(t)=\dfrac{\alpha+\beta e^{\frac{(\beta-\alpha)t}{\sqrt[p]{p-1}}}}{{1+e^{\frac{(\beta-\alpha)t}{\sqrt[p]{p-1}}}}}.
	$$
\end{theorem}

The plan of the work is as follows: In Sect. \ref{s3}, we explore minimization arguments to study the existence of a solution to problem \eqref{007}, while in Sect. \ref{s4} we prove Theorems \ref{T1} and \ref{T2}. Sect. \ref{s5} is reserved for the proof of Theorem \ref{T3}. In Sect. \ref{s6}, we make some additional comments about the extensions of these ideas. Finally, in Appendix \ref{s2} we present a brief account about some results involving Orlicz and Orlicz-Sobolev spaces for unfamiliar readers with the topic.


\section{Existence of solution}\label{s3}

In this section, we will employ the minimization technique to find a solution to the quasilinear problem \eqref{007} under conditions $(\phi_1)$-$(\phi_3)$ on $\phi$ and $(V_1)$-$(V_4)$ on $V$. Moreover, we will establish some qualitative properties that will be useful in the course of work.  The idea is looking for a minimum of the action functional
$$
F(u)=\int_{\mathbb{R}}\left(\Phi(|u'|)+V(u)\right)dt
$$
on the class of admissible function given by
$$
\Gamma_\Phi(\alpha)=\left\{u\in W^{1,\Phi}_{\text{loc}}(\mathbb{R})\mid u(-\infty)=-\alpha\text{ and }u(+\infty)=\alpha\right\}.
$$
Note that the function $\psi_{\alpha}:\mathbb{R}\rightarrow\mathbb{R}$ defined by
$$
\psi_{\alpha}(t)=\left\{\begin{array}{lll}
	-\alpha,&\mbox{if}\quad t\leq-\alpha,
	\\
	t,&\mbox{if}\quad -\alpha\leq t\leq \alpha,
	\\
	\alpha,&\mbox{if}\quad \alpha\leq t,
\end{array}\right.	
$$
belongs to $\Gamma_\Phi(\alpha)$ and $F(\psi_{\alpha})<+\infty$. Hence, since $F$ is not negative, the number
$$
c_\Phi(\alpha)=\inf_{u\in\Gamma_\Phi(\alpha)}F(u)
$$
belongs to $[0,+\infty)$. 

The set formed by the minimum points $u$ of the functional $F$ on $\Gamma_\Phi(\alpha)$ satisfying the initial condition $u(0)=0$ and $|u| \leq \alpha$ will be denoted by 
$$
K_{\Phi}(\alpha)=\left\{ u\in\Gamma_\Phi(\alpha):F(u)=c_\Phi(\alpha),~|u|\leq\alpha\text{ on $\mathbb{R}$ and }u(0)=0 \right\}.
$$
To understand the relation between $K_{\Phi}(\alpha)$ and $\Gamma_{\Phi}(\alpha)$ we observe that for any $u\in\Gamma_{\Phi}(\alpha)$ there is $v\in\Gamma_{\Phi}(\alpha)$ such that $F(v)\leq F(u)$, $v(0)=0$ and $v(s)<0< v(t)$ for all $s< 0$ and $t> 0$. Indeed, by the geometry of $u$, there are two numbers $s_0\leq t_0$ such that $u(s_0)=u(t_0)=0$, $u(s)<0$ for all $s< s_0$ and $u(t)>0$ for any $t> t_0$. So, defining the function
$$
v(t)=\left\{\begin{array}{ll}
	u(t+t_0),&\mbox{if}\quad t\geq0,
	\\
	u(t+s_0),&\mbox{if}\quad t\leq0,
	\\
\end{array}\right.	
$$ 
we have that $v\in\Gamma_{\Phi}(\alpha)$, $v(0)=0$, $v>0$ on $(0,+\infty)$, $v<0$ on $(-\infty,0)$ and $F(v)\leq F(u)$. Then we have
\begin{lemma}\label{l1}
	It holds that $K_\Phi(\alpha)$ is not empty. Moreover, any $q\in K_\Phi(\alpha)$ is a weak solution of the quasilinear equation
	\begin{equation} \label{eq}
		-\left(\phi(|u'|)u'\right)'+V'(u)=0\quad\text{in}\quad\mathbb{R}.
	\end{equation}
\end{lemma}
\begin{proof}
Let $(q_n) \subset \Gamma_{\Phi}(\alpha)$ a minimizing sequence for $F$. Without loss of generality, we can assume that $(q_n)$ satisfies the following property: 
$$
q_n(0)=0 \quad \mbox{and} \quad -\alpha\leq q_n(s)< 0< q_n(t)\leq \alpha,\text{ for all }s<0\text{ and }t>0.
$$
From this, it is easy to see that $(q_n)$ is bounded in $W^{1,\Phi}_{\text{loc}}(\mathbb{R})$. Now, in view of Lemma \ref{A4}, $W^{1,\Phi}(\mathcal{O})$ is reflexive Banach spaces whenever $\mathcal{O}$ is bounded in $\mathbb{R}$, and so, by employing a classical diagonal argument there are $q \in W^{1,\Phi}_{\text{loc}}(\mathbb{R})$ and a subsequence of $(q_n)$, still denoted by $(q_n)$, such that 
$$
q_n\rightharpoonup q\text{ in }W^{1,\Phi}(I)~\text{ and }~q_n\rightarrow q\text{ in }L^{\infty}(I),
$$
for all interval $I=(a,b)\subset\mathbb{R}$ with $-\infty<a<b<+\infty$. Therefore, $q(0)=0$, $-\alpha\leq q(s)\leq 0\leq q(t)\leq\alpha$ for all $t\geq0$ and $s\leq 0$ and $F(q)\leq c_{\Phi}(\alpha)$. We claim that $q(t)\rightarrow\alpha$ as $t\rightarrow+\infty$. Let us first prove that 
\begin{equation}\label{4}
	\displaystyle\liminf_{t\rightarrow+\infty}|q(t)-\alpha|=0.
\end{equation}
Before proceeding with the proof, we would like to point out that by $(V_1)$-$(V_3)$, it is standard to prove that there exist $\overline{a}_1, \overline{a}_2>0$ satisfying 
\begin{equation}\label{v4}
	\overline{a}_1\Phi(|t-\alpha|)\leq V(t)\leq \overline{a}_2\Phi(|t-\alpha|),~\forall t\in[0,\alpha].
\end{equation}
If \eqref{4} is not true there exist $\tilde{t},r>0$ such that $r\leq |q(t)-\alpha|$ for all $t\geq \tilde{t}$. Thus, since $\Phi$ is increasing on $(0,+\infty)$ and $q(t)\in [0,\alpha]$ for any $t\geq 0$, by \eqref{v4}, one has 
$$
\overline{a}_1\Phi(r)\leq V(q(t)),~~\forall t\geq \tilde{t},
$$
and so, 
$$
F(q)\geq \int_{t_0}^{t}V(q(t))dt\geq \overline{a}_1\Phi(r)(t-\tilde{t}) \quad \forall t>\tilde{t},
$$
which is absurd because $F(q)<+\infty$. Next we are going to show that 
\begin{equation}\label{5}
	\displaystyle \limsup_{t\rightarrow+\infty}|q(t)-\alpha|=0.
\end{equation}
If the limit above does not occur, there is $r>0$ such that $\displaystyle\limsup_{t\rightarrow+\infty}|q(t)-\alpha|>2r$. Now, fixing $\epsilon>0$ satisfying $\epsilon<\min\{r,\alpha\}$, by continuity of $q$ we can find a sequence of disjoint intervals $(\sigma_i,\tau_i)$ with $0<\sigma_i<\tau_i<\sigma_{i+1}<\tau_{i+1}$, $i\in\mathbb{N}$, and $\sigma_i\rightarrow+\infty$ as $i\rightarrow+\infty$ such that for each $i$,
$$
q(\sigma_i)=\alpha-\frac{\epsilon}{2}~~\text{and}~~q(\tau_i)=\alpha-\epsilon.
$$
Repeating the same arguments as in \cite[Lemma 2.4]{Alves2}, this gives the existence of a positive constant $\mu_{\epsilon}$ independent of $\sigma_i$ and $\tau_i$ such that 
$$
F(q)\geq\sum_{i=1}^{+\infty}\int_{\sigma_i}^{\tau_i}\left(\Phi(|q'|)+V(q)\right)dt\geq\sum_{i=1}^{+\infty}\mu_{\epsilon}=+\infty,
$$
which again contradicts the fact that $F(q)<+\infty$. Therefore, by \eqref{4} and \eqref{5}, we conclude that $q(+\infty)=\alpha$. Likewise, we can employ a similar argument to show that $q(-\infty)=-\alpha$ resulting in $F(q)=c_\Phi(\alpha)$, and so, $K_\Phi(\alpha)$ is not empty. Finally, a standard argument (see for example \cite{Alves}) ensures that $q$ is a weak solution of \eqref{eq}, and the proof is completed.
\end{proof}

The next result provides $C^{1,\gamma}$ regularity information for weak solutions of \eqref{eq} belonging to $K_\Phi(\alpha)$.

\begin{lemma}\label{l2}
	If $q\in K_\Phi(\alpha)$, then $q$ belongs to $C^{1,\gamma}_{\text{loc}}(\mathbb{R})$ for some $\gamma\in(0,1)$ and $q$ satisfies pointwise \eqref{eq}.
\end{lemma}
\begin{proof}
First of all, note that the assumption $(\phi_2)$ allows us to use a regularity result developed by Lieberman \cite[Theorem 1.7]{Lieberman} to obtain that $q\in C^{1,\gamma}_{\text{loc}}(\mathbb{R})$ for some $\gamma\in(0,1)$ whenever $q\in K_\Phi(\alpha)$. On the other hand, any weak solution $q$ of \eqref{eq} in $K_\Phi(\alpha)$ satisfies 
\begin{equation}\label{Eq}
	\left(\phi(|q'|)q'\right)'=V'(q) \text{ almost everywhere on } \mathbb{R}.
\end{equation} 
Since the right-hand side of \eqref{Eq} is a continuous function on $\mathbb{R}$, the Lemma \ref{A6} combines with \eqref{Eq} to get $\phi(|q'|)q'\in W^{1,\tilde{\Phi}}_{\text{loc}}(\mathbb{R})$, and consequently from \eqref{emb}, $\phi(|q'|)q'\in W^{1,1}_{\text{loc}}(\mathbb{R})$. Now, see for instance \cite[Theorem 8.2]{Brezis} to conclude that
$$
-\left(\phi(|q'(t)|)q'(t)\right)'+V'(q(t))=0~~\forall t\in\mathbb{R},
$$
finishing the proof. 
\end{proof}

Based on the Harnack-type inequalities found in \cite[Theorem 1.1]{Trudinger}, we will see in the next lemma that any weak solution $q \in K_\Phi(\alpha)$ of \eqref{eq} does not reach $\pm\alpha$.

\begin{lemma}\label{l3}
	If $q\in K_\Phi(\alpha)$ then $-\alpha<q(s)<0<q(t)<\alpha$ for all $s<0$ and $t>0$. 
\end{lemma}
\begin{proof}
The conditions $(\phi_3)$ and $(V_4)$ permits to apply the Hanarck-type inequality found in \cite{Trudinger} to prove that $|q|<\alpha$,  see also \cite[Lemma 2.6]{Alves2} for more details.  In order to show that  $0<q(t)$ for all $t>0$, let us assume by contradiction that there is $t_1>0$ satisfying $q(t_1)=0$. In this case, defining the function
$$
q_1(t)=\left\{\begin{array}{ll}
	q(t),&\mbox{if}\quad t\leq 0\\
	q(t+t_1),&\mbox{if}\quad t \geq 0,
\end{array}\right.
$$
we obtain that $q_1\in \Gamma_{\Phi}(\alpha)$, and so, 
$$
c_\Phi(\alpha)\leq F(q_1)=F(q)-\int_{0}^{t_1}\left(\Phi(|q'|)+V(q)\right)dt=c_\Phi(\alpha)-\int_{0}^{t_1}\left(\Phi(|q'|)+V(q)\right)dt,
$$
from where it follows that 
$$
\int_{0}^{t_1}\left(\Phi(|q'|)+V(q)\right)dt=0.
$$
Thus, $q=\alpha$ on $(0,t_1)$, which a contraction. Similarly, $q(t)<0$ for any $t<0$, and the proof is over.  
\end{proof}

As a consequence the elements of $K_\Phi(\alpha)$ are increasing.

\begin{lemma}\label{14part1}
	If $q\in K_\Phi(\alpha)$, then $q$ is increasing on $\mathbb{R}$.
\end{lemma}
\begin{proof}
Indeed, if this does not occur, then we can assume without loss of generality that there are $t_1,t_2\in(0,+\infty)$ with $t_1<t_2$ and $q(t_2)\leq q(t_1)$. Thereby, from Intermediate Value Theorem, there exists $t_0\in(0,t_1]$ satisfying $q(t_0)=q(t_2)$. With that in mind, setting the function
$$
\tilde{q}(t)=\left\{\begin{array}{ll}
	q(t),&\mbox{if}\quad t\leq t_0\\
	q(t+t_2-t_0),&\mbox{if}\quad t_0\leq t,
\end{array}\right.
$$
it is simple to check that $\tilde{q}\in \Gamma_\Phi(\alpha)$ and  
$$
c_\Phi(\alpha)\leq F(\tilde{q})=F(q)-\int_{t_0}^{t_2}\left(\Phi(|q'|)+V(q)\right)dt=c_\Phi(\alpha)-\int_{t_0}^{t_2}\left(\Phi(|q'|)+V(q)\right)dt,
$$
implying that 
$$
\int_{t_0}^{t_2}\left(\Phi(|q'|)+V(q)\right)dt=0
$$
and so $q(t)=\alpha$ for all $t\in (t_0,t_2)$, contrary to Lemma \ref{l3}.	
\end{proof}

The proofs of the next two lemmas are similar to those developed in \cite{Alves1}, but we present the proof in detail for the reader's convenience.

\begin{lemma}\label{14part2}
	If $q\in K_\Phi(\alpha)$, then $q'$ is non-increasing on $[0,+\infty)$. In particular, $q$ is concave in $[0,+\infty)$.
\end{lemma}
\begin{proof}
	In fact, as $q(t)\in[0,\alpha]$ for all $t\geq0$, $(V_4)$ provides $V'(q(t))\leq 0$ for each $t\in [0,+\infty)$, and hence, from \eqref{eq}, we get the inequality $\left(\phi(|q'(t)|)q'(t)\right)'\leq 0$ for all $t\in [0,+\infty)$, and so, the function $t\in[0,+\infty) \mapsto\phi(q'(t))q'(t)$ is non-increasing on $[0,+\infty)$. From this, by $(\phi_1)$, it is easy to check that $q'$ is non-increasing on $[0,+\infty)$, and the lemma is proved. 
\end{proof} 

\begin{lemma}\label{l4}
 If $q\in K_\Phi(\alpha)$, then $q'(t)>0$ for all $t\in\mathbb{R}$.
\end{lemma}
\begin{proof}
	If there exists $t_0\geq0$ such that $q'(t_0)=0$, it follows from Lemmas \ref{14part1} and \ref{14part2} that $0\leq q'(t)\leq q'(t_0)=0$ for any $t\geq t_0$, which leads to $q'=0$ on $[t_0,+\infty)$, that is, $q$ is constant on $[t_0,+\infty)$. But, since $q(t)$ goes to $\alpha$ as $t$ goes to $+\infty$, we must have $q(t)=\alpha$ for all $t\in [t_0,+\infty)$, which again contradicts Lemma \ref{l3}. Therefore, $q'>0$ on $[0,+\infty)$. Similarly, $q'>0$ on $(-\infty,0]$, and the lemma follows.
\end{proof}

It is well known that one of the main difficulties in working with quasilinear elliptic equations is the lack of $C^2$ regularity. However, one-dimensional case permits to show that the weak solutions have good qualitative properties. For example, in the next lemma we will show the existence of second derivative for the solutions of \eqref{eq} that belongs to $K_\Phi(\alpha)$.

\begin{lemma}\label{l5}
	If $q\in K_\Phi(\alpha)$, then $q$ is twice differentiable on $\mathbb{R}$.
\end{lemma}
\begin{proof}
For any $q \in K_\Phi(\alpha)$, we can apply \cite[Theorem 8.2]{Brezis} together with Lemma \ref{l1} to obtain  
$$
\phi(|q'(t)|)q'(t)=\int_{0}^{t}V'(q(s))ds+	\phi(|q'(0)|)q'(0).
$$
On the other hand, by virtue of condition $(\phi_1)$, it follows from Inverse Function Theorem that the function $g(t)=\phi(t)t$ has an inverse $g^{-1}$, defined on the interval $I=g(J)$ with $J=(0,+\infty)$, which is differentiable on $I$. Consequently, since $q'>0$ on $\mathbb{R}$, one has  
$$
q'(t)=g^{-1}\left(\int_{0}^{t}V'(q(s))ds+	\phi(|q'(0)|)q'(0)\right),~~\forall t\in\mathbb{R}.
$$
Thereby, since $V\in C^1(\mathbb{R},\mathbb{R})$, the last equality ensures that $q'$ is differentiable on $\mathbb{R}$, finishing the proof.	
\end{proof}



\section{The proofs of Theorems \ref{T1} and \ref{T2}}\label{s4}

In this section, the proof of Theorems \ref{T1} and \ref{T2} will be carried out. To this end, from now on we will assume that $q$ is a solution of the heteroclinic quasilinear elliptic problem 
\begin{equation}\label{NewEquation}
	-\left(\phi(|u'|)u'\right)'+V'(u)=0~~\text{on}~~\mathbb{R}~~\text{and}~~u(\pm\infty)=\pm\alpha
\end{equation}
belonging to $C^{1,\gamma}_{\text{loc}}(\mathbb{R})$ for some $\gamma\in(0,1)$.

Let us first establish some qualitative properties on $q$.

\begin{lemma}\label{NewLemma1}
 It turns out that $q'(\pm\infty)=0$.
\end{lemma}
\begin{proof}
If the limit $q'(+\infty)=0$ is not valid, then there exist $r>0$ and a sequence $(t_n)\subset\mathbb{R}$ with $t_n\rightarrow+\infty$ satisfying 
\begin{equation}\label{NewEq1}
	|q'(t_n)|\geq r~~ \forall n\in\mathbb{N}.	
\end{equation}
So, by continuity, given $t\in[t_n,t_n+1]$ there exists $s_n\in[t_n,t_n+1]$ such that 
$$
\left|\phi(|q'(t)|)q'(t)-\phi(|q'(t_n)|)q'(t_n)\right|\leq \left|\left(\phi(|q'(s_n)|)q'(s_n)\right)'\right|=\left|V'(q(s_n))\right|.
$$
Consequently,  
$$
\left|\phi(|q'(t)|)q'(t)-\phi(|q'(t_n)|)q'(t_n)\right|\to 0\text{ as }n\to+\infty.
$$
Therefore, there exists $n_0\in\mathbb{N}$, which is uniform for all $t \in [t_n,t_n+1]$, satisfying
$$
\left|\phi(|q'(t)|)q'(t)-\phi(|q'(t_n)|)q'(t_n)\right|< \frac{\phi(r)r}{2}~~\forall n\geq n_0. 
$$
As $\phi(t)t$ is increasing on $(0,+\infty)$ we obtain by \eqref{NewEq1} that 
$$
\phi(r)r-\phi(|q'(t)|)| q'(t)|\leq \frac{\phi(r)r}{2}~~\forall n\geq n_0,
$$
that is, 
\begin{equation}\label{NewEq2}
	\frac{\phi(r)r}{2}\leq \phi(|q'(t)|)|q'(t)|~\text{ for }t\in[t_n,t_n+1]\text{ and }n\geq n_0.
\end{equation}
On the other hand, for each $n\in\mathbb{N}$ there is $\bar t_n\in [t_n,t_n+1]$ such that 
$$
|q(t_n+1)-q(t_n)|=|q'(\bar t_n)|,
$$
and therefore $|q'(\bar t_n)|\to 0$ as $n\to+\infty$. Thus, given $\epsilon>0$ such that $\phi(\epsilon)\epsilon<\dfrac{\phi(r)r}{4}$ then there exists $n_1\geq n_0$ satisfying $|q'(\bar t_n)|\leq \epsilon$ for all $n\geq n_1$. Consequently, 
$$
\phi(|q'(\bar t_n)|)|q'(\bar t_n)|\leq \phi(\epsilon)\epsilon,
$$
which contradicts \eqref{NewEq2}. Similarly, $q'(-\infty)=0$, and the lemma follows.
\end{proof}

Let us now use the condition of strict convexity on $V$ around $\pm\alpha$ given in $(V_5)$ to prove the following local uniqueness result.

\begin{proposition}\label{L:minimilocal} 
Denote $A^{+}=[0,+\infty)$ and $A^{-}=(-\infty,0]$. For any $y_{0}\in B_{\rho}(\pm\alpha)=(\pm\alpha-\rho,\pm\alpha+\rho)$ there exists a unique solution $q^{\pm}_{y_{0}}$ of 
	\begin{equation}\label{NewProblem1}
		-\left(\phi(|u'|)u'\right)'+V'(u)=0\quad\text{on}\quad A^{\pm}
	\end{equation}
	such that
	\begin{equation*}\label{e:+}
		q^{\pm}_{y_{0}}(0)=y_{0},~q^{\pm}_{y_{0}}(t)\in B_{\rho}(\pm\alpha)\hbox{ for any }t\in A^{\pm}~\text{and}~q^{\pm}_{y_{0}}(+\infty)=\pm\alpha.
	\end{equation*}
	Moreover, 
	$$
	\|q_{y_{0}}^{\pm}-\pm\alpha\|_{L^{\infty}(A^{\pm})}= |y_{0}-\pm\alpha|
	$$
	and defining 
	$$
	F_{\pm}(u)=\int_{A^\pm}\left(\Phi(|u'|)+V(u)\right)dt,
	$$ 
	the function $q_{y_{0}}^{\pm}$ is the (unique) minimum of $F_{\pm}$ on the set
	\[
	\Gamma_{\Phi}(y_0,\pm\alpha)=\left\{u\in W^{1,\Phi}_{\text{loc}}(A^{\pm})\mid u(\pm\infty)=\pm\alpha,~u(0)=y_{0},~\text{and}~u(t)\in B_{\rho}(\pm\alpha)~\text{on}~A^{\pm}\right\}.
	\]
\end{proposition}
\begin{proof}
We consider the case that involves the interval $A^{+}$ and the case $A^{-}$ can be treated similarly. Let be $y_{0}\in B_{\rho}(+\alpha)$ and $q_m$ a minimal heteroclinic solution of \eqref{eq}. Thus, there exists $t_{1}>0$ such that $q_{m}(t_{1})=y_{0}$ and $q_{m}(t)\in B_{\rho}(+\alpha)$ for any $t>t_{1}$. Now, setting
$$
\tilde q_{m}(t)=q_{m}(t+t_{1})~~\text{for}~~t\in\mathbb{R}
$$
we have $\tilde q_{m}\in \Gamma_{\Phi}(y_0,+\alpha)$. We claim that $\tilde{q}_m$ minimizes $F_{+}$ on $\Gamma_{\Phi}(y_0,+\alpha)$. Indeed, given any $u\in \Gamma_{\Phi}(y_0,+\alpha)$, the function 
$$
\bar u(t)=\left\{\begin{array}{ll}
	u(t)&\mbox{if}\quad t\in A^+,\\
	\tilde{q}_m(t)&\mbox{if}\quad t\in \mathbb{R}\setminus A^{+}
\end{array}\right.
$$ 
belongs to $\Gamma_\Phi(\alpha)$. Consequently, since $F(\tilde{q}_m)=F(q_m)$ we derive
$$
0\leq F(\bar u)-F(\tilde q_{m})=F_{+}(u)-F_{+}(\tilde q_{m}),
$$ 
which gives $F_{+}(\tilde q_{m})\leq F_{+}(u)$ demonstrating that $\tilde q_{m}$ minimizes $F_{+}$ on $\Gamma_{\Phi}(y_0,+\alpha)$. To conclude the proof, we have to show the uniqueness part of the statement. First observe that by $(V_5)$ $F_{+}$ is strictly convex on $\Gamma_{\Phi}(y_0,+\alpha)$ from which follows the uniqueness of the minimum $\tilde q_{m}$. Then if we assume that $v$ is another solution of \eqref{NewProblem1} such that $v(0)=y_{0}$, $v(t)\in B_{\rho}(+\alpha)$ for any $t\in A^{+}$ and $v(+\infty)=\alpha$, necessarily $v$ is not a minimum of $F_{+}$ on $\Gamma_{\Phi}(y_0,+\alpha)$. Consequently,
\[
F_{+}(q)=\int_{A^+}\left(\Phi(|v'|)+V(v)\right)dt>F_{+}(\tilde q_{m})=\int_{A^+}\left(\Phi(|\tilde{q}_{m}'|)+V(\tilde q_{m})\right)dt,
\]
and thereby there exist $t_{+}>0$ and $\mu_{+}>0$ such that
\begin{equation}\label{eq:zero1}
	\int_{0}^{\bar t}\left(\Phi(|v'|)+V(v)\right)dt-\int_{0}^{\bar t}\left(\Phi(|\tilde q_{m}'|)+V(\tilde q_{m})\right)dt\geq \mu_{+}\hbox{ for any }\bar t\geq t_{+}.
\end{equation}
Now, by convexity of $V$ on $B_\rho(+\alpha)$,
\[
V(v(t))-V(\tilde q_{m}(t))\leq V'(v(t))\left(v(t)-\tilde q_{m}(t)\right)\text{ for all }t\in A^{+}
\]
and by Lemma \ref{A6}-(c), 
$$
\Phi(|v'(t)|)-\Phi(|\tilde{q}_m'(t)|)\leq \phi(|v'(t)|)v'(t)\left(v'(t)-\tilde{q}_m'(t)\right)\text{ for all }t\in A^{+}.
$$
Using these estimates in \eqref{eq:zero1}, we get 
\begin{equation}\label{eq:uno1}
	\int_{0}^{\bar t}\left(\phi(|v'|)v'(v'-\tilde{q}_m')+V'(v)(v-\tilde q_{m})\right)dt\geq \mu_{+}\hbox{ for any }\bar t\geq t_{+}.
\end{equation}
On the other hand, for each $\bar t>t_{+}$ we define the function
\[
v_{\bar t}(t)=\begin{cases} \tilde q_{m}(t)&\hbox{ for }t>\bar t+1,\\
		v(\bar t)+(t-\bar t)(\tilde q_{m
		}(\bar t+1)-v(\bar t))& \hbox{ for }\bar t\leq t \leq \bar t+1,\\
		v(t)& \hbox{for }0\leq t<\bar t.\end{cases}
\] 
Note that the function $\varphi_{\bar t}=v_{\bar t}-\tilde q_{m}$ has support in $[0,\bar t+1]$ and since $v$ solves \eqref{NewProblem1} we have 
\begin{equation}\label{eq:due1}
	\int_{0}^{+\infty}\left(\phi(|v'|)v'\varphi_{\bar t}'+V'(v)\varphi_{\bar t}\right)dt=0\hbox{ for any }\bar t\geq t_{+}.
\end{equation}
As $\varphi_{\bar t}=v-\tilde q_{m}$ on $(0,\bar t)$ we may write
\begin{equation}\label{eq:due2}
	\begin{split}
		\int_{0}^{+\infty}\left(\phi(|v'|)v'\varphi_{\bar t}'+V'(v)\varphi_{\bar t}\right)dt=\int_{0}^{\bar t}&\left(\phi(|v'|)v'(v'-\tilde q_{m}')+V'(v)(v-\tilde q_{m})\right)dt\\&+\int_{\bar t}^{\bar t+1}\left(\phi(|v'|)v'\varphi_{\bar t}'+V'(v)\varphi_{\bar t}\right)dt.
	\end{split}
\end{equation}
Now, the argument from Lemma \ref{NewLemma1} works to show that $v'(+\infty)=0$ and so $\phi(|v'(t)|)v'(t)\to0$ as $t\to+\infty$. Thus, since $v_{\bar t}'(t)=\tilde q_{m}(\bar t+1)-v(\bar t)$ for $t\in(\bar t,\bar t+1)$ and $v(+\infty)=\alpha$ we can fix $\bar t>t_{+}$ such that
\begin{equation}\label{eq:tre1}
	\left|\int_{\bar t}^{\bar t+1}\left(\phi(|v'|)v'\varphi_{\bar t}'+V'(v)\varphi_{\bar t}\right)dt\right|\leq\frac{\mu_{+}}{2}.
\end{equation}
For such value of $\bar t$, by \eqref{eq:uno1}, \eqref{eq:due2} and \eqref{eq:tre1} we recover
\[
\int_{0}^{+\infty}\left( \phi(|v'|)v'\varphi_{\bar t}'+V'(v)\varphi_{\bar t}\right)dt\geq \frac{\mu_{+}}{2}
\]
contradicting \eqref{eq:due1} and so proving the Lemma.	
\end{proof}

Thanks to the above proposition, we are ready to show that the problem \eqref{NewEquation} has a unique solution up to translation, namely the minimal heteroclinic solution $q_m$ of the equation \eqref{eq}.

\begin{lemma}\label{NewLemma2}
	It turns out that $q$ is a translation of $q_m$. 
\end{lemma}
\begin{proof}
First, let us fix $y_0=-\alpha+\frac{\rho}{2}$. Since $q(\pm\infty)=\pm\alpha$, there exists $t_{0}\in\mathbb{R}$ such that $q(t_{0})=y_0$ and $|q(t)-\alpha|<\frac{\rho}{2}$ for any $t<t_{0}$. If we consider the function 
$$
\tilde q(t)=q(t+t_{0}),~~t\in\mathbb{R},
$$
then it satisfies
\[
\tilde q(0)= y_{0}\in B_{\rho}(-\alpha),~\tilde q(t)\in B_{\rho}(-\alpha)\hbox{ for any }t\in A^{-}=(-\infty,0]\hbox{ and }\tilde{q}(-\infty)=-\alpha.
\]
Consequently, from Proposition \ref{L:minimilocal} we have that $\tilde q=q^{-}_{y_{0}}$ on $A^{-}$. We repeat the reasoning with the minimum point $q_{m}$ choosing $t_{1}\in\mathbb{R}$ such that $q_{m}(t_{1})=y_{0}$ and $q_{m}(t)\in B_{\rho}(-\alpha)$ for any $t<t_{1}$, we have that the function 
$$
\tilde q_{m}(t)=q_{m}(t+t_{1}),~~t\in\mathbb{R},
$$
satisfies $\tilde q_m=q^{-}_{y_{0}}$ on $A^{-}$ and therefore
$$
\tilde q(t)=\tilde q_{m}(t)~~\text{for all}~~t\in A^{-}.
$$ 
Now, we define the set
$$
X=\left\{s\geq0\mid\tilde{q}(t)=\tilde{q}_m(t)\text{ for all }t\leq s\right\}.
$$
Clearly $X$ is not empty, since $0\in X$. Setting $\eta\equiv\sup X$, our goal is to prove that $\eta=+\infty$. Indeed, if not, then $\eta\in\mathbb{R}$, and thus from the continuity of the functions $\tilde{q}$ and $\tilde{q}_m$ it is easy to see that $\eta\in X$. Now, since $q_m\in C^{1,\gamma}_{\text{loc}}(\mathbb{R})$ for some $\gamma\in(0,1)$ and $\tilde{q}=\tilde{q}_m$ on $(-\infty,\eta]$ we have that $\tilde{q}\in C^{1,\gamma}_\text{loc}((-\infty,\eta))$, $\tilde{q}(\eta)=\tilde{q}_m(\eta)$ and $\tilde{q}'(\eta)=\tilde{q}_m'(\eta)$. Moreover, as $q_m'>0$ on $\mathbb{R}$, we find $\rho_0,\epsilon>0$ such that
$$
\tilde{q}'(t)\geq\rho_0~\text{ for all }~t\in [\eta-\epsilon,\eta+\epsilon].
$$
By Cauchy problem in quasilinear configurations \cite[Theorem 3.14]{Alves1} we can conclude that $\tilde{q}=\tilde{q}_m$ on $[\eta-\epsilon,\eta+\epsilon]$, and do, $\eta+\epsilon\in X$, which is impossible. Therefore, $\tilde{q}(t)=\tilde{q}_m(t)$ for all $t\in\mathbb{R}$, and this ends the proof. 
\end{proof}

So far we have not used condition $(\phi_4)$ on $\phi$, but we will now see in the next two results that $(\phi_4)$ applies an important rule to show that $q$ and its derivative $q'$ behave exponentially at infinity. 

\begin{lemma}\label{L1}
	There exist $\theta_1,\theta_2,\theta_3,\theta_4>0$ such that 
	\begin{equation}\label{expoest}
		0<\alpha-q(t)\leq \theta_1e^{-\theta_2t}\text{ for all }t\geq 0
	\end{equation}
	and 
	\begin{equation}\label{expoest1}
		0<\alpha+q(t)\leq \theta_3e^{-\theta_4t}\text{ for all }t\leq 0.
	\end{equation}
\end{lemma}
\begin{proof}
We will only demonstrate \eqref{expoest}, since \eqref{expoest1} follows similarly. Since $q(+\infty)=\alpha$, given $\delta\in (0,\alpha)$, there exists $R_{\delta}>0$ such that 
$$
\alpha-\delta\leq q(t)<\alpha~\text{ for all }~t\geq R_{\delta}.
$$
For each $R>R_\delta$, let us define the following functions
$$
q_R(t)=\left\{\begin{array}{ll}
	q(t),&\mbox{if}\quad t\in[R_{\delta},R]\\
	q(2R-t),&\mbox{if}\quad t\in[R,2R-R_{\delta}]
\end{array}\right.
$$ 
and
$$
u_R(t)=\alpha-q_R(t) ~\text{ for }~ t\in[R_{\delta},2R-R_{\delta}].
$$
From $(V_4)$, it is easily seen that
\begin{equation}\label{20}
	V'(q_R)\leq-\frac{c_1}{c_2}(\alpha-\delta)\phi(c_2 u_R)(c_2 u_R)~~\text{on}~~[R_{\delta},R].
\end{equation}
On the other hand, by direct calculations, for any
\begin{equation}\label{203}
 \psi\in W_0^{1,\Phi}\left((R_{\delta},2R-R_{\delta})\right),~\psi(t)\geq 0\text{ and }\psi(t)=\psi(2R-t)\text{ for every }t\in[R_{\delta},R],
\end{equation}
we have
\begin{equation}\label{21}
	\int_{R_{\delta}}^{2R-R_{\delta}}\left(\phi\left(|u_R'|\right)u_R'\psi'-V'(q_R)\psi\right)dt=0. 	
\end{equation}
Thus, \eqref{20} combined with \eqref{21} gives 
\begin{equation}\label{202}
	\int_{R_{\delta}}^{2R-R_{\delta}}\left(\phi\left(|u_R'|\right)u_R'\psi'+\frac{c_1}{c_2}(\alpha-\delta)\phi(c_2 u_R)c_2 u_R\psi\right)dt\leq 0.
\end{equation}
In order to get the exponential estimate (\ref{expoest}), let us consider the following real function
$$
\zeta(t)=\delta\frac{\cosh(a(t-R))}{\cosh(a(R-R_{\delta}))},~~t\in\mathbb{R},
$$ 
where $a>0$ is a constant such that
\begin{equation}\label{204}
	a<\min\left\{c_2,\frac{c_1(\alpha-\delta)}{c_2m}\right\}.
\end{equation}
Now, we claim that
\begin{equation}\label{01}
	\left(\phi(|\zeta'(t)|)\zeta'(t)\right)'\leq ma\phi(|\zeta'(t)|)c_2\zeta(t),~~\forall t\in\mathbb{R}.
\end{equation} 
Indeed, first of all note that for values of $t\in\mathbb{R}$ such that $|\zeta'(t)|>0$ it is easily seen that
$$
\left(\phi(|\zeta'(t)|)\zeta'(t)\right)'=\phi(|\zeta'(t)|)\zeta''(t)+\phi'(|\zeta'(t)|)|\zeta'(t)|\zeta''(t).
$$
Since $\zeta''(t)=a^2\zeta(t)$ for all $t\in\mathbb{R}$, gathering the last equality with $(\phi_2)$ we get   
$$
\left(\phi(|\zeta'(t)|)\zeta'(t)\right)'\leq m a^2\phi(|\zeta'(t)|)\zeta(t)\leq m a\phi(|\zeta'(t)|)c_2\zeta(t).
$$
On the other hand, for the case $\zeta'(t)=0$, an easy verification shows that
$$
\left(\phi(|\zeta'(t)|)\zeta'(t)\right)'=\phi(0)\zeta''(t)\leq m a\phi(|\zeta'(t)|)c_2\zeta(t).
$$
Hence, in all cases \eqref{01} is true. Consequently, since  $|\zeta'(t)|\leq a\zeta(t)$ for any $t\in\mathbb{R}$, by $(\phi_4)$ one has
$$
\phi(|\zeta'(t)|)\leq \phi(c_2\zeta(t)),
$$
which together with \eqref{01} ensures that
$$
-\left(\phi(|\zeta'(t)|)\zeta'(t)\right)'+m a\phi(c_2\zeta(t))c_2\zeta(t)\geq 0,~~\forall t\in\mathbb{R}.
$$ 
Applying a direct calculation to the inequality above, one gets
\begin{equation}\label{205}
	\int_{R_{\delta}}^{2R-R_{\delta}}\left(\phi(|\zeta'|) \zeta'\psi'+ma\phi(c_2\zeta)c_2\zeta\psi\right)dt\geq 0,
\end{equation} 
where $\psi$ is given as in \eqref{203}. Finally, comparing the functions $\zeta$ and $u_R$ at $R_\delta$ and $2R-R_\delta$, we see that 
$$
u_R(R_{\delta})\leq \delta=\zeta(R_{\delta}) ~\text{ and }~u_R(2R-R_{\delta})\leq \delta=\zeta(2R-R_{\delta}).
$$
Then, the function $\psi_R=(u_R-\zeta)^+$ belongs to class \eqref{203} and from \eqref{202}, \eqref{204}, \eqref{205} we deduce
\begin{equation}\label{02}
	\int_{R_{\delta}}^{2R-R_{\delta}}\left(\left(\phi(|{u_R}'|){u_R}'-\phi(|{\zeta}'|){\zeta}'\right){\psi_R}'+am(\phi(c_2 u_R)c_2 u_R-\phi(c_2\zeta)c_2\zeta)\psi_R\right)dt\leq 0.
\end{equation}
Now, setting 
$$
\mathcal{O}_R=\left\{t\in[R_{\delta},2R-R_{\delta}]:u_R(t)\geq \zeta(t)\right\},
$$ 
and using Lemma \ref{A5} -(b) in inequality \eqref{02}, we infer that $\psi_R=0$ on $\mathcal{O}_R$, that is, 
$$
u_R(t)\leq\zeta(t) \text{ for any }t\in[R_{\delta},2R-R_{\delta}].
$$
Recalling that $\frac{e^t}{2}\leq\cosh(t)\leq e^t$ for all $t\geq 0$, it follows that 
$$
\alpha-q(t)\leq 2\delta e^{a(R_{\delta}-t)}~~\forall t\in[R_{\delta},R].
$$
Finally, since $R\geq R_\delta$ is arbitrary, we arrive at 
$$
\alpha-q(t)\leq 2\delta e^{a(R_{\delta}-t)}~~\forall t\geq R_\delta,
$$
which is enough to guarantee the validity of \eqref{expoest}, showing the desired result. 
\end{proof}

The reader is invited  to see that the condition $(\phi_4)$ plays an important role in the last lemma. However, making a simple inspection of the proof, it is possible to see that $(\phi_4)$ can be substituted by the following condition:  
\begin{equation}\label{phi}
	\phi(|\zeta'(t)|)\leq c\phi(\zeta(t))~\text{ for all }~t\in\mathbb{R},
\end{equation}
for some $c>0$. Hence, we can state the following result that will be used in the proof of Theorem \ref{T3}, see Section 5.

\begin{corollary}
	Assume $(\phi_1)$-$(\phi_2)$ and that $\phi$ satisfies \eqref{phi}. Then, the solution $q$ of \eqref{007} has the exponential behaviors \eqref{expoest} and \eqref{expoest1}.
\end{corollary}

From Lemma \ref{L1} we can now also characterize the asymptotic behaviour of $q'$.

\begin{lemma}\label{L2}
	There exist $\beta_1,\beta_2,\beta_3,\beta_4>0$ such that 
	\begin{equation}\label{exp3}
		0<q'(t)\leq\beta_1e^{-\beta_2t}~\text{ for all }~ t\geq0
	\end{equation}
	and 
	\begin{equation}\label{exp4}
		0<q'(t)\leq\beta_3e^{\beta_4t}~\text{ for all }~ t\leq0.
	\end{equation}
\end{lemma}
\begin{proof}
We will only prove \eqref{exp3}, because the estimate \eqref{exp4} follows analogously. First of all,  note that by $(V_4)$ there is a constant $w>0$ such that 
\begin{equation*}
	|V'(q(t))|\leq w|\alpha-q(t)|\text{ for all }t\geq 0.
\end{equation*}
Thereby, for any $x,t\in[0,+\infty)$ with $t<x$, one has
$$
\left|\int_{t}^{x}\left(\phi(q'(s))q'(s)\right)'ds\right|\leq w\int_{t}^{x}|\alpha-q(s)|ds.
$$ 
Now, since $\phi(|q'|)q'\in W^{1,1}_{\text{loc}}(\mathbb{R})$, 
$$
\left|\phi(q'(t))q'(t)-\phi(q'(x))q'(x)\right|\leq w\int_{t}^{x}|\alpha-q(s)|ds.
$$
Consequently, from Lemma \ref{L1}, 
$$
|\phi(q'(t))q'(t)-\phi(q'(x))q'(x)|\leq w\theta_1\int_{t}^{x}e^{-\theta_2s}ds=w\theta_1\left(-\frac{1}{\theta_2}e^{-\theta_2x}+\frac{1}{\theta_2}e^{-a_\theta t}\right).
$$
Since $q'(x)\to 0$ as $x\to+\infty$, it then follows
$$
\phi(q'(t))q'(t)\leq \frac{w\theta_1}{\theta_2}e^{-\theta_2t}~\text{ for all }~t\geq0.
$$
To conclude the estimate, we observe that gathering  Lemma \ref{A2} with \eqref{0} we obtain
$$
l\Phi(1)\frac{\xi_0(t)}{t}\leq \phi(t)t,~~\forall t>0.
$$
Thus, since there exists $t_0>0$ such that $0<q'(t)\leq 1$ for any $t\geq t_0$, the inequality above together with definition of $\xi_0$ ensures that 	 
$$
l\Phi(1)\left(q'(t)\right)^{m-1}\leq \phi(q'(t))q'(t)\leq \frac{w\theta_1}{\theta_2}e^{-\theta_2t},~~\forall t\geq t_0,
$$
and thereby, 
$$
q'(t)\leq d_1e^{-d_2t},~~\forall t\geq t_0,
$$
where the constants $d_1$ and $d_2$ are given by
$$
d_1=\left(\frac{w\theta_1}{l\Phi(1)\theta_2}\right)^{\frac{1}{m-1}}~\text{ and }~d_2=\frac{\theta_2}{m-1}.
$$
Finally, as $q'$ is a continuous function, it is possible to find positive real numbers $\beta_1$ and $\beta_2$ satisfying precisely the assertion of the lemma.
\end{proof}

Now our next step is to show that when $V$ is the $\Phi$-double well potential, that is, $V(t)=\Phi(|t^2-\alpha^2|)$, the unique solution $q$ of \eqref{007} behaves similarly to the solution $q_+(t)=\alpha\tanh(\alpha t)$ of classical logistic heteroclinic problem
$$
-u''+2u^3-2\alpha^2u=0~~\text{in}~~\mathbb{R},~~u(0)=0,~~u(\pm\infty)=\pm\alpha,
$$
in the sense that it is between dilations of $q_+$, on which it depends on the parameters $l$ and $m$.
For technical simplicity, let us consider the following function $\varphi:(0,+\infty)\to\mathbb{R}$ defined by 
$$
\varphi(t)=\frac{\phi'(t)t+\phi(t)}{\phi(t)}, \quad \forall t>0. 
$$
From $(\phi_2)$,
\begin{equation}\label{110}
	l-1\leq\varphi(t)\leq m-1,~~\forall t>0. 
\end{equation} 
In addition, let us define the function $G:[0,+\infty)\to [0,+\infty)$ by
\begin{equation}\label{NewEq.}
	G(t)=\int_{0}^{t}\phi(s)s\varphi(s)ds,
\end{equation} 
which is of class $C^1$ on $(0,+\infty)$ and satisfies 
\begin{equation}\label{ineq}
(l-1)\Phi(t)\leq G(t)\leq (m-1)\Phi(t),~~\forall t\geq0.	
\end{equation}
With these settings, we can prove the following result.

\begin{lemma}\label{L3}
	It turns out that $F(q)<+\infty$. Moreover, there are $C_1,C_2>0$ such that 
	$$
	C_2F(q)\leq \int_{\mathbb{R}}\left(G(q')+V(q)\right)dt\leq C_1F(q).
	$$
\end{lemma}
\begin{proof}
$F(q)<+\infty$ occurs trivially since $q$ is actually a minimal heteroclinic solution. To finish the proof, thanks to \eqref{ineq} one has  
$$
C_2F(q)\leq \int_{\mathbb{R}}\left(G(q')+V(q)\right)dt\leq C_1F(q),
$$
where $C_1=\max\{1,m-1\}$ and $C_2=\min\{l-1,1\}$. This ends the proof.
\end{proof}

For purposes of this section, let us introduce the {\it energy function} 
$$
E(q)(t)=G(q'(t))-V(q(t)),~~t\in\mathbb{R}.
$$
In the sequel, we will see that the energy is conserved along the trajectory $(q,q')$.

\begin{lemma}\label{L4}
	It turns out that $E(q)(t)=0$ for all $t\in\mathbb{R}$.
\end{lemma}
\begin{proof}
Since $q$ is twice differentiable, 
\begin{equation}\label{1}
	\frac{d}{dt}E(q)(t)=\phi(q'(t))q'(t)\varphi(q'(t))q''(t)-V'(q(t))q'(t),~~\forall t\in\mathbb{R}.
\end{equation}
On the other hand,
\begin{equation}\label{1'}
		\left(\phi(q'(t))q'(t)\right)'=\phi'(q'(t))q'(t)q''(t)+\phi(q'(t))q''(t)=\varphi(q'(t))\phi(q'(t))q''(t),~~\forall t\in\mathbb{R}.,
\end{equation}
Now, combining \eqref{1'} with the fact that $q$ satisfies the equation
$$
\left(\phi(q'(t))q'(t)\right)'=V'(q(t)),~~\forall t\in\mathbb{R},
$$ 
we arrive at
\begin{equation}\label{3}
	\varphi(q'(t))\phi(q'(t))q''(t)-V'(q(t))=0~\text{ for }~t\in\mathbb{R}.
\end{equation}
From \eqref{1} and \eqref{3}, $\frac{d}{dt}E(q)(t)=0$ for any $t\in\mathbb{R}$, and thereby, there is $\lambda\in\mathbb{R}$ such that $E(q)(t) =\lambda$ for all $t\in\mathbb{R}$. According to the Lemma \ref{L3}, 
$$
\int_{\mathbb{R}}\left(G(q')+V(q)\right)dt<+\infty,
$$
and then, there exists $(t_n)\subset\mathbb{R}$ with $t_n\to+\infty$ such that 
$$
G(q'(t_n))\to0 \text{ and }V(q(t_n))\to 0.
$$
From that,
$$
E(q)(t_n)=G(q'(t_n))-V(q(t_n))\to0,
$$
resulting in $\lambda=0$, and the proof is completed.
\end{proof}

As a direct consequence of the Lemma \ref{L4}, 
$$
G(q'(t))=V(q(t))\text{ for all }t\in\mathbb{R}.
$$
On the other hand, by virtue of $G'(t)=t\left(\phi(t)t\right)'>0$ for any $t>0$, 
it follows from the Inverse Function Theorem that $G$ has an inverse $G^{-1}$, defined on the interval $I=G(J)$, where $J=(0,+\infty)$, which is differentiable in $I$. Thus, we can write 
\begin{equation}\label{expr}
	q'=G^{-1}(V(q))~\text{ on }~\mathbb{R}.
\end{equation}
From (\ref{expr}), it is clear that $q$ satisfies (up to translations) the following Cauchy problem
\begin{equation*}
	\left\{\begin{array}{ll}
		y'=f(y)&\mbox{on}\quad \mathbb{R},
		\\
		y(0)=0,
	\end{array}\right.\eqno{(CP)}
\end{equation*}
where the function $f:(-\alpha,\alpha)\to\mathbb{R}$ is defined by $f(t)=G^{-1}(V(t)).$ Finally, the framework we have just presented puts us in a position to show our main result.

\subsection{Proof of Theorem \ref{T1}} \mbox{}\\
The previously developed study ensures that $q$ is the unique solution of \eqref{007}, up to the translation, and with the help of condition $(\phi_4)$ the functions $q$ and $q'$ satisfy the exponential decay estimates given in item $(a)$ of Theorem \ref{T1}. The last part of the proof boils down in proving item $(b)$ of Theorem \ref{T1}. So, when $V(t)=\Phi(|\alpha^2-t^2|)$, one has
$$
G(q'(t))=\Phi(\alpha^2-q(t)^2),~~~\forall t\in\mathbb{R}.
$$
The above inequality together with \eqref{ineq} gives  
$$
(l-1)\Phi(q'(t))\leq \Phi(\alpha^2-q(t)^2)\leq (m-1)\Phi(q'(t)),~~~\forall t\in\mathbb{R}.
$$
So, if $l-1\leq 1\leq m-1$, the convexity of $\Phi$ leads to  
$$
\Phi\left((l-1)q'(t)\right)\leq \Phi(\alpha^2-q(t)^2)\leq \Phi\left((m-1)q'(t)\right),~~~\forall t\in\mathbb{R}.
$$
Hence, since $\Phi^{-1}$ is an increasing function, one gets
$$
(l-1)q'(t)\leq \alpha^2-q(t)^2\leq (m-1)q'(t).
$$
From this,    
$$
\int\dfrac{dq}{(\alpha-q)(\alpha+q)}\leq\frac{1}{l-1}\int dt,
$$
and so, 
$$
ln\left(\frac{\alpha+q}{\alpha-q}\right)\leq\frac{2\alpha }{l-1}t, 
$$
that is, 
$$
\frac{\alpha+q}{\alpha-q}\leq e^{ct},\text{ where }
c=\frac{2\alpha }{l-1}.
$$
Therefore, 
$$
q(t)\leq \alpha\left(\frac{e^{ct}-1}{e^{ct}+1}\right)=\alpha\tanh\left(\frac{c}{2}t\right),
$$
or equivalently, 
$$
q(t)\leq \alpha\tanh\left(\frac{\alpha  t}{l-1}\right) ~\text{ for all }~t\in\mathbb{R}.
$$
Similarly, 
$$
\alpha\tanh\left(\frac{\alpha  t}{m-1}\right)\leq q(t)~\text{ for all }~t\in\mathbb{R},
$$
proving the item $(i)$. Now, when $l-1>1$, one finds
$$
\Phi\left(q'(t)\right)\leq \Phi(\alpha^2-q(t)^2)\leq \Phi\left((m-1)q'(t)\right),~~~\forall t\in\mathbb{R}.
$$
The previous argument implies that
$$
\alpha\tanh\left(\frac{\alpha  t}{m-1}\right)\leq q(t)\leq \alpha \tanh(\alpha t)~\text{ for all }~t\in\mathbb{R},
$$
which proves $(ii)$. Case $(iii)$ can be established by a similar argument, and the proof is completed.
\hspace{13,3 cm} $\Box$

From the previous study, we see that the unique solution $q$ of problem \eqref{007} when $V(t)=\Phi(|\alpha^2-t^2|)$ behaves similarly to the function  $t\mapsto\alpha\tanh(t)$. However, when $\phi(t)=t^{p-2}$ for $p\in(1,+\infty)$, Theorem \ref{T2} tells us that this is exactly what happens.

\subsection{Proof of Theorem \ref{T2}} \mbox{}\\
The proof is done by an explicit computation. Indeed, we will henceforth denote $\phi(t)=t^{p-2}$ for all $t>0$ with $p\in(1,+\infty)$. In this case, it is easy to check that 
$$
\Phi(t)=\frac{|t|^p}{p},~~G(t)=(p-1)\Phi(t)~~~\text{and}~~V(t)=\frac{|t^2-\alpha^2|^p}{p}.
$$
Hence,
$$
E(q)(t)=(p-1)\Phi(q'(t))-\Phi(\alpha^2-q(t)^2),
$$
where $q$ is the unique solution of \eqref{007}. Thus, since $E(q)(t)=0$ for all $t\in\mathbb{R}$, one has
$$
\Phi(q'(t))=\frac{b}{p-1}\Phi\left(\alpha^2-q(t)^2\right),~~\forall t\in\mathbb{R}.
$$
But, as $\Phi^{-1}(t)=(pt)^{\frac{1}{p}}$ for any $t>0$, it becomes evident that 
$$
q'(t)=\left(\frac{1}{p-1}\right)^{\frac{1}{p}}\left(\alpha^2-q(t)^2\right), ~~\forall t\in\mathbb{R}.
$$
Finally, notice that the argument given in the proof of Theorem \ref{T1} works well to establish the equality below 
$$
q(t)=\alpha\tanh\left(\frac{\alpha  t}{(p-1)^{\frac{1}{p}}}\right),
$$
which is precisely the assertion of the theorem.
\hspace{7.1 cm} $\Box$

\begin{remark}
Now, we would like to point out that the expression \eqref{expr} yields a procedure to find the explicit solution of problem \eqref{007}. However, this procedure depends strongly of the function $\phi$. Moreover, when $\phi(t)=t^{p-2}$, it is well known that $$G(t)=(p-1)\dfrac{t^p}{p},$$ but when $\phi$ satisfies  $(\phi_2)$ with $l=m$, we extract from \eqref{ineq} the following equality 
	$$
	G(t)=(m-1)\Phi(t),~~~\forall t\geq 0.
	$$
	However, this only happens when $\phi$ is a power type function with exponent equal to $m-2$. Indeed, by $(\phi_2)$, it is easy to see that $\phi$ satisfies the following ordinary differential equation
	$$
	y'-\frac{m-2}{t}y=0~~\text{ on }~~(0,+\infty).
	$$ 
	Then, if the initial condition $\phi(1)=c$ is prescribed, one has 
	$$
	\phi(t)=ct^{m-2},~~t>0.
	$$ 
\end{remark}

We conclude this section with some comments based on what was introduced and discussed earlier. We start by presenting the following result.

\begin{proposition}\label{p1}
Let $u$ be a solution of
\begin{equation}\label{008}
 	-\left(\phi(|u'|)u'\right)'+V'(u)=0~\text{ in }~\mathbb{R},~~u(0)=0.
\end{equation}
Then, $u$ is heteroclinic from $-\alpha$ to $\alpha$ if, and only if, $u$ has finite energy.
\end{proposition}
\begin{proof}
It was discussed earlier that, under the condition that $u$ connects $-\alpha$ and $\alpha$, we must have $F(u)<+\infty$. Conversely, if $u$ has finite energy, the reader can see that this implies $E(u)(t)=0$ for all $t\in\mathbb{R}$, from which it follows that $u$ satisfies the Cauchy problem $(CP)$, and so, due to the uniqueness of the solution, $u$ can only be a heteroclinic solution from $-\alpha$ to $\alpha$, and thus the theorem is proved.
\end{proof}

\begin{remark}
	Note that the condition $u(0)=0$ in problem \eqref{008} excludes the stationary solutions $-\alpha$ and $\alpha$ which, in turn, have finite energies. Finally, we would also like to point out that the information proved above is totally related to the fact that the energy function $E(u)(t)$ conserves energy, that is, $E(u)(t)=0$ for all $t\in\mathbb{R}$. Out of curiosity, the reader can carefully observe that any solution of 
	$$
	-\left(\phi(|u'|)u'\right)'+V'(u)=0~\text{ in }~\mathbb{R},~~u(0)=0,
	$$
	satisfies $E(u)(t)=\lambda$ for any $t\in\mathbb{R}$ and for some real number $\lambda$. Thereby, to understand a little bit how is the behavior of solutions that conserve energy, we can explore the variation of $\lambda$. For example, if $\lambda>0$, then we can write 
	$$
	u'(t)=G^{-1}\left(\lambda+V(u(t))\right).
	$$ 
	Since $G^{-1}$ is increasing, we derive that $u(t)\geq G^{-1}(\lambda)t$ for all $t\in\mathbb{R}$, and so, $u$ blows up in infinite time. In \cite{Li}, the reader will find a more precise explanation of these ideas for the case $\Phi(t)=\frac{|t|^2}{2}$.
\end{remark}


\section{The proof of Theorem \ref{T3}}\label{s5}

This section concerns the proof of Theorem \ref{T3}. First of all, we are going to make a truncation on the prescribed mean curvature operator to obtain an auxiliary quasilinear problem of the form \eqref{007}. More precisely, for each $L>0$, we consider the following quasilinear problem 
$$
-\left(\phi_L\left(|u'|\right)u'\right)'+V'(u)=0~\text{ in }~\mathbb{R},~u(0)=0,~u(\pm\infty)=\pm\alpha,\eqno{(AP)_L}
$$
where $\phi_L:[0,+\infty)\rightarrow[0,+\infty)$ is defined by
$$
\phi_L(t)=\left\{\begin{array}{lll}
	\dfrac{1}{\sqrt{1+t^2}},&\mbox{if}\quad t\in[0,\sqrt{L}],
	\\
	x_L(t^2-L-1)^2+y_L,&\mbox{if}\quad t\in[\sqrt{L},\sqrt{L+1}],
	\\
	y_L,&\mbox{if}\quad t\in[\sqrt{L+1},+\infty)
\end{array}\right.
$$
with
$$
x_L=\frac{\sqrt{1+L}}{4(1+L)^2}~\text{ and }~y_L=(4L+3)x_L.
$$
Our goal now is to insert $(AP)_L$ in the configurations of the previous sections. Have this in mind, we will introduce the function 
$$
\Phi_L(t)=\int_0^{|t|}\phi_L(s)sds.
$$

The following result will be used frequently in this section and its proof can be found in \cite[Lemma 3.1]{Alves2}.

\begin{lemma}\label{R51}
	For each $L>0$, the functions $\phi_L$ and $\Phi_L$ satisfy the following properties:
	\begin{itemize}
		\item [(a)] $\phi_L$ is $C^1$.
		\item [(b)] $y_L\leq \phi_L(t)\leq 1$ for each $t\geq 0$.
		\item [(c)] $\dfrac{y_Lt^2}{2}\leq \Phi_L(t)\leq \dfrac{t^2}{2}$ for any $t\geq0$.
		\item [(d)] $\Phi_L$ is a convex function. 
		\item [(e)] $\left(\phi_L\left(t\right)t\right)'>0$ for all $t>0$.
	\end{itemize}
\end{lemma}

The last lemma ensures that $\Phi_L$ is an $N$-function that behaves similarly to the quadratic function $|t|^2$. Furthermore, it was also shown in  \cite[Lemma 3.2]{Alves2} that the space  $L^{\Phi_L}$ is exactly the Lebesgue space $L^2$ and the norm of $L^{\Phi_L}$ is equivalent to the usual norm of $L^2$, which makes $L^{\Phi_L}$ a reflexive space. Still on the functions $\phi_L$ and $\Phi_L$, we also have the following result.

\begin{lemma}\label{R52}
	It turns out that $\phi_L$ satisfies $(\phi_1)$-$(\phi_3)$. Moreover, the best constants in $(\phi_2)$ are 
	$$
	l_L=1+\inf_{t>0}\frac{\left(\phi_L(t)t\right)'}{\phi_L(t)}=\left\{\begin{array}{ll}
		\frac{6L+6-L^2}{4L+4},&\mbox{if}\quad L\leq 1,
		\\
		\frac{7+4L}{4L+4},&\mbox{if}\quad L>1
	\end{array}\right.\text{ and }~~m_L=1+\sup_{t>0}\frac{\left(\phi_L(t)t\right)'}{\phi_L(t)}=2.
	$$ 
\end{lemma}
\begin{proof}
Thanks to Lemma \ref{R51} clearly $(\phi_1)$ and $(\phi_3)$ are verified with $q=2$. To finish the proof, note that 
$$
\phi'_L(t)=\left\{\begin{array}{lll}
	-\dfrac{t}{(\sqrt{1+t^2})^{3}},&\mbox{if}\quad t\in[0,\sqrt{L}],
	\\
	4x_L(t^2-L-1)t,&\mbox{if}\quad t\in[\sqrt{L},\sqrt{L+1}],
	\\
	0,&\mbox{if}\quad t\in[\sqrt{L+1},+\infty)
\end{array}\right.
$$ 
and 
\begin{equation}\label{r0}
	\frac{(\phi_L(t)t)'}{\phi_L(t)}=1+\frac{\phi_L'(t)t}{\phi_L(t)}~\text{ for all }~t\in\mathbb{R}.
\end{equation}
It is easy to see that
\begin{equation}\label{r1}
	\frac{(\phi_L(t)t)'}{\phi_L(t)}=1~\text{ for all }~t\in(\sqrt{L+1},+\infty).
\end{equation}
Moreover, when $t\in[0,\sqrt{L}]$,  
$$
\frac{\phi'_L(t)t}{\phi_L(t)}=\frac{-t^2}{1+t^2},
$$ 
from where it follows that this function is decreasing, and thereby,
\begin{equation}\label{r2}
	\frac{1}{L+1}\leq\frac{(\phi_L(t)t)'}{\phi_L(t)}\leq1~\text{ for all }~ t\in[0,\sqrt{L}].
\end{equation}
Now, for $t\in[\sqrt{L},\sqrt{L+1}]$, we see that
\begin{equation}\label{r00}
	\frac{\phi'_L(t)t}{\phi_L(t)}=\frac{4(t^2-L-1)t^2}{(t^2-L-1)^2+4L+3}.
\end{equation}
We are going to study the behavior of the function $f(t)=4(t^2-L-1)t^2$ on $[\sqrt{L},\sqrt{L+1}]$. For this, let us note that
$$
f'(t)=16t\left(t-\sqrt{\frac{L+1}{2}}\right)\left(t+\sqrt{\frac{L+1}{2}}\right).
$$
Let us now analyze the infimum and supremum of the function \eqref{r00} on $[\sqrt{L},\sqrt{L+1}]$ in cases.

If $L\leq 1$, the critical point $\sqrt{\frac{L+1}{2}}$ of $f$ is a minimum point in $[\sqrt{L},\sqrt{L+1}]$, and so,  
$$
-(L+1)^2=f\left(\sqrt{\frac{L+1}{2}}\right)\leq f(t)\leq f\left(\sqrt{L+1}\right)=0 ~~ \forall t\in[\sqrt{L},\sqrt{L+1}].
$$
Gathering \eqref{r0} and \eqref{r00},
$$
1+\frac{-(L+1)^2}{(t^2-L-1)^2+4L+3}\leq\frac{(\phi_L(t)t)'}{\phi_L(t)}\leq 1~\text{ for all }~ t\in[\sqrt{L},\sqrt{L+1}],
$$
and from
\begin{equation}\label{r5}
	4L+3\leq (t^2-L-1)^2+4L+3\leq 4L+4~~\text{for any}~~t\in[\sqrt{L},\sqrt{L+1}],
\end{equation}
we arrive at 
\begin{equation}\label{r3}
	\frac{2L-L^2+2}{4L+4}\leq\frac{(\phi_L(t)t)'}{\phi_L(t)}\leq 1~\text{ for all }~ t\in[\sqrt{L},\sqrt{L+1}].
\end{equation}

If $L>1$, the function $f$ is increasing on $[\sqrt{L},\sqrt{L+1}]$, and in this case,
$$
-4L=f\left(\sqrt{L}\right)\leq f(t)\leq f\left(\sqrt{L+1}\right)=0 ~~ \forall t\in[\sqrt{L},\sqrt{L+1}],
$$
and hence, from \eqref{r0}, \eqref{r00} and \eqref{r5}, 
\begin{equation}\label{r4}
	\frac{3}{4L+4}\leq\frac{(\phi_L(t)t)'}{\phi_L(t)}\leq 1~\text{ for all }~ t\in[\sqrt{L},\sqrt{L+1}].
\end{equation}
Finally, the lemma follows from \eqref{r1}, \eqref{r2}, \eqref{r3} and \eqref{r4}.
\end{proof}

Although $\Phi_L$ behaves like a quadratic function, the function $\phi_L$ does not satisfy $(\phi_4)$, because it is clearly non-increasing on $\mathbb{R}$. However, in order to use the arguments developed in the previous sections, we will show below that $\phi_L$ checks the inequality \eqref{phi}.

\begin{lemma}\label{R53}
	For each $L>0$, the function $\phi_L$ satisfies the inequality \eqref{phi}.
\end{lemma}
\begin{proof}
According to Lemma \ref{R51}-$(b)$,  
$$
y_L\leq \phi_L(|\zeta'(t)|), \phi_L(\zeta(t))\leq 1,~~\forall t\in\mathbb{R},
$$
and so, 
$$
\phi_L(|\zeta'(t)|)\leq \frac{1}{y_L}\phi_L(\zeta(t)),~~\forall t\in\mathbb{R},
$$
which completes the proof.
\end{proof}

The framework built so far on the truncated operator of prescribed mean curvature allows us to use the same arguments explored in Sections \ref{s3} and \ref{s4} for problem \eqref{010}. Thereby, in order to clarify the ideas, let us assume here that the potential $V$ belongs to $C^2(\mathbb{R},\mathbb{R})$ and that satisfies the assumptions $(V_1)$-$(V_2)$ and $(V_6)$-$(V_7)$. We claim at this moment that $V$ satisfies $(V_3)$ with $\Phi_L$ and $(V_4)$ with $\phi_L$. Indeed, by $(V_6)$ there are $\rho_1,\rho_2>0$ and $\theta\in (0,\frac{\alpha}{2})$ such that
\begin{equation}\label{***}
	\rho_1|t-\alpha|^2\leq V(t)\leq \rho_2|t-\alpha|^2,~\forall t\in(\alpha-\theta,\alpha+\theta).
\end{equation}
Then, by Lemma \ref{R51}-(c) and \eqref{***},
$$
2\rho_1\Phi_L(|t-\alpha|)\leq V(t)\leq \frac{2\rho_2}{y_L}\Phi_L(|t-\alpha|),~\forall t\in(\alpha-\theta,\alpha+\theta).	
$$	
Similarly, decreasing $\theta$ if necessary, there are $\rho_3,\rho_4>0$ such that 
$$
2\rho_3\Phi_L(|t+\alpha|)\leq V(t)\leq \frac{2\rho_4}{y_L}\Phi_L(|t+\alpha|),~\forall t\in(-\alpha-\theta,-\alpha+\theta),
$$	
showing that $V$ satisfies $(V_3)$ with $\Phi_L$. Moreover, combining item $(b)$ of Lemma \ref{R51} with $(V_7)$, it is easy to see that $V$ also satisfies $(V_4)$ with $\phi_L$. Finally, by $(V_6)$ we have that $V$ is strictly convex in the neighborhoods of global minima $\pm\alpha$ of $V$, from which $(V_5)$ follows. Therefore, we can directly apply the ideas from the proof of Theorem \ref{T1} to establish the result below.

\begin{proposition}\label{R54}
Assume $V\in C^2(\mathbb{R},\mathbb{R})$, $(V_1)$-$(V_2)$, $(V_6)$-$(V_7)$. For each $L>0$, the problem $(AP)_L$ has a unique twice differentiable solution $q_{\alpha,L}$ in $C^{1,\gamma}_{\text{loc}}(\mathbb{R})$ for some $\gamma\in (0,1)$ satisfying the following exponential decay estimates
\begin{equation*}
	0<\alpha-q(t)\leq\theta_1e^{-\theta_2t}~~\text{ and }~~q'(t)\leq \beta_1e^{-\beta_2t}~~\text{ for all }t\geq0
\end{equation*}
and 
\begin{equation*}
	0<\alpha+q(t)\leq \theta_3e^{\theta_4t}~~\text{ and }~~q'(t)\leq \beta_3e^{\beta_4t}~~\text{ for all }t\leq0.
\end{equation*}
Moreover, if $V(t)=\left(t^2-\alpha^2\right)^2$ then $q_{\alpha,L}$ has the following estimates
\begin{equation}\label{th}
	\alpha\tanh\left(\alpha\sqrt{2} t\right)\leq q_{\alpha,L}(t)\leq \alpha\tanh\left(\frac{\alpha\sqrt{2}t}{\sqrt{y_L(l_L-1)}}\right)~\text{ for }~t\geq 0.
\end{equation}
\end{proposition}
\begin{proof}
The proof of existence, uniqueness and exponential-type estimates can be done following the same approach explored in Sections \ref{s3} and \ref{s4}. Now, to show estimate \eqref{th}, we recall that the arguments in Sect. \ref{s4} give 
$$
(l_L-1)\Phi_L(q'(t))\leq \left(\alpha^2-q(t)^2\right)^2\leq \Phi_L(q'(t)),~~\forall t\in\mathbb{R}.
$$
Using Lemma \ref{R51}-(c) in the inequality above one gets
$$
\frac{y_L(l_L-1)}{2}q'(t)^2\leq \left(\alpha^2-q(t)^2\right)^2\leq \frac{1}{2}q'(t)^2,~~\forall t\in\mathbb{R},
$$
or equivalently,
$$
\frac{\sqrt{y_L(l_L-1)}}{\sqrt{2}}q'(t)\leq\alpha^2-q(t)^2\leq\frac{1}{\sqrt{2}}q'(t),~~\forall t\in\mathbb{R}.
$$
Arguing as in the proof of Theorem \ref{T1}, we establish the desired estimate, and the proof is over. 
\end{proof}

We would like to point out here that there is no restriction for the global minimum $\alpha$ of the potential $V$ in the above theorem, that is, Theorems \ref{R54} holds for all real number $\alpha>0$. Furthermore, our interest in this section is to find a solution $q_{\alpha,L}$ that satisfies the following inequality
\begin{equation} \label{EQTL} 
	\|q'_{\alpha,L}\|_{L^{\infty}(\mathbb{R})}\leq \sqrt{L},
\end{equation} 
because it clearly ensures that $q_{\alpha,L}$ is a solution to the prescribed mean curvature operator. To achieve this objective, it is necessary to obtain a control involving the root $\alpha$ of $V$, and in this point, the condition $(V_8)$ applies an important rule in our work, because it allows us to use the theory of elliptic regularity found in \cite{Lieberman} to show that there is a solution  $q_{\alpha,L}$ that satisfies the inequality \eqref{EQTL} whenever $\alpha>0$ is small enough. 

\begin{lemma}\label{R55}
	Assume $(V_1)$-$(V_2)$ and $(V_6)$-$(V_8)$. Given $L>0$ there is $\delta>0$ such that for each $\alpha\in(0,\delta)$ we have that  
	$$
	\|q_{\alpha,L}\|_{C^1([1-r,1+r])}<\sqrt{L}~\text{ for all }~r\in\mathbb{R},
	$$
	where $q_{\alpha,L}$ was given in Proposition \ref{R54}.
\end{lemma}
\begin{proof}
The argument follows the same lines as the proof of \cite[Lemma 3.3]{Alves2}, but we present the proof in detail for the reader’s	convenience. Arguing by contradiction, there are $(r_n)\subset\mathbb{R}$ and $(\alpha_n)\subset(0,+\infty)$ with $\alpha_n\rightarrow0$ as $n\rightarrow+\infty$ and
\begin{equation}\label{515}
	\|q_{\alpha_n,L}\|_{C^1[-1+r_n,1+r_n]}\geq \sqrt{L},~~~\forall n\in\mathbb{N}.
\end{equation} 
For the sake of illustration, let us define $q_n(t)=q_{\alpha_n,L}(t+r_n)\text{ for } t\in\mathbb{R}$. Then each $q_n$ is a weak solution of the equation
$$
-\left(\phi_L\left(|q'|\right)q'\right)'+B_n(t)=0,
$$
where $B_n(t)=V'(q_{n}(t))$ for $t\in\mathbb{R}$. From $(V_8)$, it is possible to prove that there is $C>0$ independent of $n$ such that 
$$
|B_n(t)|\leq C\text{ for all }t\in\mathbb{R}\text{ and }n\in\mathbb{N}.
$$
Consequently, by \cite[Theorem 1.7]{Lieberman}, $q_n\in C^{1,\beta}_{\text{loc}}(\mathbb{R})$ for some $\beta\in(0,1)$ with 
$$
\|q_n\|_{C^{1,\beta}_{\text{loc}}(\mathbb{R})}\leq R,~~~\forall~n\in\mathbb{N},
$$
for some $R>0$ independent of $n$. The above estimate allows us to use the Arzelà–Ascoli Theorem to show that there are $q\in C^1([-1,1])$ and a subsequence of $(q_n)$, still denoted by $(q_n)$, satisfying $q_n\rightarrow q$ in $C^1([-1,1])$. Now, as $\|q_{\alpha_n,L}\|_{L^\infty(\mathbb{R})} \rightarrow 0$ as $n\rightarrow+\infty$, $q=0$ on $[-1,1]$. Thus, there is $n_0\in\mathbb{N}$ such that 
$$
\|q_n\|_{C^1\left([-1,1]\right)}<\sqrt{L},~~~\forall~n\geq n_0,
$$
that is,
$$
\|q_{\alpha_n,L}\|_{C^1([-1+r_n,1+r_n])}<\sqrt{L},~~~\forall~n\geq n_0,
$$ 
which contradicts \eqref{515}.
\end{proof}

Finally, to conclude this section, let's use the framework discussed above to prove Theorem \ref{T3}. 

\noindent{\textbf{Proof of Theorem \ref{T3}.}}\\
Let us first prove the existence of a solution to \eqref{010}. To see this, note that by Lemma \ref{R55}, for each $L>0$ there is $\delta>0$ such that for each $\alpha>0$ with $\alpha\in(0,\delta)$ one has
$$
\|q'_{\alpha,L}\|_{L^{\infty}(\mathbb{R})}<\sqrt{L},
$$
from which it follows that $q_{\alpha,L}$ is a solution for $\eqref{010}$ whenever $\alpha\in(0,\delta)$. To show the uniqueness, let's assume that $q_1$ and $q_2$ are solutions to the problem \eqref{010} that are twice differentiable in $C^{1,\gamma}_{\text{loc}}(\mathbb{R})$ for some $\gamma\in(0,1)$. Now we would like to point out that the same arguments found in Sect. \ref{s4} work to show that $q'_1,q'_2\in L^{\infty}(\mathbb{R})$. Thus, there exists $L>0$ such that  
$$
|q'_1(t)|,|q'_2(t)|\leq \sqrt{L}~\text{ for all }t\in\mathbb{R},
$$
from where it follows that $q_1$ and $q_2$ are solutions of the problem $(AP)_L$. Therefore, by Proposition \ref{R54}, $q_1=q_2$ on $\mathbb{R}$. Finally, if $q$ is the unique solution of \eqref{010}, then clearly $q$ is the unique solution of the auxiliary problem $(AP)_L$ with $L=\|q'\|_{L^\infty(\mathbb{R})}$ and therefore inequality \eqref{th0} is guaranteed by \eqref{th} with 
$$
\kappa=\sqrt{y_{\|q'\|_{\infty}}\left(l_{\|q'\|_{\infty}}-1\right)},
$$ and the proof is over. \hspace{11.5 cm} $\square$

In our argument, condition $(V_8)$ was used only to guarantee the existence of a solution to the prescribed mean curvature problem \eqref{010} whenever $\alpha>0$ is small, but to show the uniqueness of solution, this control over $\alpha$ it was not necessary. Therefore, we can immediately deduce the following result.

\begin{corollary}\label{R56}
 	Assume $V\in C^2(\mathbb{R},\mathbb{R})$, $(V_1)$-$(V_2)$, $(V_6)$-$(V_7)$. Under these conditions, if \eqref{010} has a solution, then it must be unique and it satisfies the inequality \eqref{th0}.
\end{corollary}


\section{Final remarks}\label{s6}

In this brief final section, we will address some results related to those mentioned in the previous sections, in addition to other general observations. We would like to start by saying that it is not important that the global minima of the potential $V$ are symmetrical. This assumption is only made for the convenience of notation. Then, all results proved in the present paper can be replicated for the case where $V$ has global minima at $\alpha$ and $\beta$ with $\alpha<\beta$. More precisely, considering the problem 
\begin{equation}\label{011}
-\left(\phi(|u'|)u'\right)'+V'(u)=0~\text{ in }~\mathbb{R},~u(0)=u_0,~u(-\infty)=\alpha,~u(+\infty)=\beta,
\end{equation}
where $u_0\in(\alpha,\beta)$ and $V$ satisfies hypotheses $(V_1)$,
\begin{itemize}
	\item[$(\tilde{V}_2)$] $V(\alpha)=V(\beta)=0$ and $V(t)>0$ for all $t\in(\alpha,\beta)$,
	\item[$(\tilde{V}_3)$] There exist $\delta\in(0,\frac{\beta-\alpha}{2})$ and $a_i>0$ such that 
	$$
	a_1\Phi(|t-\beta|)\leq V(t)\leq a_2\Phi(|t-\beta|),~\forall t\in(\beta-\delta,\beta]
	$$
	and 
	$$
	a_3\Phi(|t-\alpha|)\leq V(t)\leq a_4\Phi(|t-\alpha|),~\forall t\in[\alpha,\alpha+\delta),
	$$
	\item[$(\tilde{V}_4)$] There exist $\tilde{t}\in(\alpha,\beta)$ and $c_i>0$ such that
	$$
	-c_3(t-\tilde{t})\phi(c_4|t-\beta|)|t-\beta|\leq V'(t)\leq-c_1(t-\tilde{t})\phi(c_2|t-\beta|)|t-\beta|,~~\forall t\in[\tilde{t},\beta]
	$$
	and 
	$$
	-c_7(t-\tilde{t})\phi(c_8|t-\alpha|)|t-\alpha|\leq V'(t)\leq -c_5(t-\tilde{t})\phi(c_6|t-\alpha|)|t-\alpha|,~~~\forall t\in[\alpha,\tilde{t}],
	$$
	\item[$(\tilde{V}_5)$] There exists $\rho>0$ such that $V$ is strictly convex on $(\alpha-\rho,\alpha+\rho)$ and on $(\beta-\rho,\beta+\rho)$.
\end{itemize}

As a prototype of a potential for which $(V_1)$, $(\tilde{V}_2)$-$(\tilde{V}_5)$ all are satisfied, consider 
\begin{equation}\label{newV}
	V(t)=\Phi\left(|(t-\alpha)(t-\beta)|\right).
\end{equation}
To recap, given $(V_1)$ and $(\tilde{V}_2)$-$(\tilde{V}_5)$ there is an analogue of Theorem \ref{T1}, and in this case we would like to write it as follows.

\begin{theorem}
Assume $(\phi_1)$-$(\phi_3)$, $(V_1)$, $(\tilde{V}_2)$-$(\tilde{V}_5)$. Then the problem \eqref{011} has a unique twice differentiable solution $q$ in $C^{1,\gamma}_{\text{loc}}(\mathbb{R})$ for some $\gamma\in(0,1)$. Moreover, 
		\begin{itemize}
		\item[(a)] If $(\phi_4)$ occurs then there are $\theta_i,\beta_i>0$ such that
		\begin{equation}\label{Newexp1}
			0<\beta-q(t)\leq\theta_1e^{-\theta_2t}~~\text{ and }~~0<q'(t)\leq \beta_1e^{-\beta_2t}~~\text{ for all }t\geq0
		\end{equation}
		and 
		\begin{equation}\label{Newexp2}
			0<q(t)-\alpha\leq \theta_3e^{\theta_4t}~~\text{ and }~~0<q'(t)\leq \beta_3e^{\beta_4t}~~\text{ for all }t\leq0.
		\end{equation}
		\item[(b)] When $V$ is the potential given in \eqref{newV}, then $q$ has the following estimates
		\begin{itemize}
			\item[(i)] if $l-1\leq1\leq m-1$, then 
			$$
			\dfrac{\beta e^{\frac{(\beta-\alpha)t}{m-1}}+\alpha}{e^{\frac{(\beta-\alpha)t}{m-1}}+1}\leq q(t)\leq\dfrac{\beta e^{\frac{(\beta-\alpha)t}{l-1}}+\alpha}{e^{\frac{(\beta-\alpha)t}{l-1}}+1}~\text{ for }~t\geq 0,
			$$ 
			\item[(ii)] if $l-1>1$, then 
			$$
			\dfrac{\beta e^{\frac{(\beta-\alpha)t}{m-1}}+\alpha}{e^{\frac{(\beta-\alpha)t}{m-1}}+1}\leq q(t)\leq \dfrac{\beta e^{(\beta-\alpha)t}+\alpha}{e^{(\beta-\alpha)t}+1}~\text{ for }~t\geq 0,
			$$
			\item[(iii)] if $m-1<1$, then 
			$$
			\dfrac{\beta e^{(\beta-\alpha)t}+\alpha}{e^{(\beta-\alpha)t}+1}\leq q(t)\leq\dfrac{\beta e^{\frac{(\beta-\alpha)t}{l-1}}+\alpha}{e^{\frac{(\beta-\alpha)t}{l-1}}+1}~\text{ for }~t\geq 0. $$
		\end{itemize}
\end{itemize}
\end{theorem}

There is a similar version of Theorem \ref{T2} and it was stated in the introduction, namely Theorem \ref{T4}. The version of Theorem \ref{T3} follows below.

\begin{theorem}
	Assume $V\in C^2(\mathbb{R},\mathbb{R})$, $(V_1)$, $(\tilde{V}_2)$. Also consider the following assumptions
	\begin{itemize}
		\item[$(\tilde{V}_6)$] $V''(\alpha),V''(\beta)>0$.
		\item[$(\tilde{V}_7)$] There are $d_1,d_2,4_3,d_4>0$ and $\tilde{t}\in(\alpha,\beta)$ such that 
		$$-d_1t|t-\beta|\leq V'(t)\leq -d_2t|t-\beta| \text{ for all }t\in[\tilde{t},\beta]$$ and $$-d_3t|t-\alpha|\leq V'(t)\leq -d_4t|t-\alpha| \text{ for all }t\in[\alpha,\tilde{t}].$$
		\item[$(\tilde{V}_8)$] There are $\lambda>0$ and $M(\lambda)>0$ such that $\displaystyle\sup_{t\in(\alpha,\beta)}|V'(t)|\leq M(\lambda)$ for all pair $(\alpha,\beta)$ satisfying $\max\{|\alpha|,|\beta|\}\in(0,\lambda)$.
	\end{itemize} 
	Then, there exists  $\lambda_0>0$ such that for each pair $(\alpha,\beta)$ with $\max\{|\alpha|,|\beta|\}\in(0,\lambda_0)$ the problem \eqref{010} has a unique twice differentiable solution $q$ in $C^{1,\gamma}_{\text{loc}}(\mathbb{R})$ for some $\gamma\in(0,1)$ satisfying the exponential decay estimates like those given in \eqref{Newexp1} and \eqref{Newexp2}. Moreover, if $V(t)=(t-\alpha)^2(t-\beta)^2$ then $q$ satisfies 
	$$
	\dfrac{\beta e^{(\beta-\alpha)\sqrt{2}t}+\alpha}{e^{(\beta-\alpha)\sqrt{2}t}+1}\leq q(t)\leq\dfrac{\beta e^{\frac{(\beta-\alpha)\sqrt{2}t}{\kappa}}+\alpha}{e^{\frac{(\beta-\alpha)\sqrt{2}t}{\kappa}}+1}~\text{ for }~t\geq 0,
	$$ 
	where $\kappa$ is a positive constant that depends on $\|q'\|_{L^{\infty}(\mathbb{R})}$.
\end{theorem}

The reader is invited to see that the problem below 
\begin{equation*}
-\left(\phi(|u'|)u'\right)'+V'(u)=0~\text{ in }~\mathbb{R},~u(0)=0,~u(-\infty)=\beta,~u(+\infty)=\alpha,
\end{equation*}
also has a unique solution under conditions $(\phi_1)$-$(\phi_3)$, $(V_1)$ and $(\tilde{V}_2)$-$(\tilde{V}_5)$, and so, we can write results analogous to Theorems \ref{T1} to \ref{T3}. However, considering the above problem under the conditions of Theorem \ref{T2}, this becomes trivial, as the unique solution in this case is $-q$, where $q$ satisfies Theorem \ref{T2}. This fact occurs every time the potential $V$ is even, because in this case we can restrict the class of admissible functions $\Gamma_{\Phi}(\alpha)$ to odd functions and apply an argument found in \cite{Alves2} to show that the solution is odd. Furthermore, the reciprocal is also true, as shown in the result below.

\begin{proposition}
	Assume that  $(\phi_1)$-$(\phi_3)$, $(V_1)$-$(V_5)$ are satisfied. If $q$ is the unique solution to problem \eqref{007} then $V$ is even in $(-\alpha,\alpha)$ if, and only if, $q$ is odd on $\mathbb{R}$. 
\end{proposition}
\begin{proof}
To begin with, let us assume that $V$ is even in $(-\alpha,\alpha)$. In this case, considering the function 
$$
\tilde{q}(t)=\left\{\begin{array}{ll}
	q(t),&\mbox{if}\quad t\geq0,
	\\
	-q(-t),&\mbox{if}\quad t\leq 0,
\end{array}\right.	
$$
it is easy to see that $F(\tilde{q})=F(q)$. Then by uniqueness of the solution we must have $q=\tilde{q}$ on $\mathbb{R}$, and so, $q$ is odd. Conversely,  assume that  $q$ is odd on $\mathbb{R}$. In this case, $q'$ is even and the equality below 
\begin{equation}\label{u}
	G(q')=V(q).
\end{equation}
yields that $V$ is even in $(-\alpha,\alpha)$. Indeed, for each $t\in(-\alpha,\alpha)$, there is $s\in\mathbb{R}$ such that $q(s)=t$, and so, from \eqref{u},
$$
V(-t)=V(-q(s))=V(q(-s))=G(q'(-s))=G(q'(s))=V(q(s))=V(t), 
$$
and the proof is over.
\end{proof}

Before concluding this section, we would like to point out that the equality \eqref{expr} also allowed us to obtain more information about the geometry of $V$, see the result below.

\begin{proposition}
	Assume that $q$ is a solution of 
	$$
	-\left(\phi(|u'|)u'\right)'+V'(u)=0~\text{ in }~\mathbb{R},~u(\pm\infty)=\pm\alpha.
	$$
Then $(V_2)$ occurs.
\end{proposition}
\begin{proof}
First, from Lemma \ref{NewLemma1}, one has $q'(\pm\infty)=0$. Thereby, since $q$ satisfies $G(q')=V(q)$ on $\mathbb{R}$, where $G$ was defined in \eqref{NewEq.}, we deduce that $G(0)=V(\pm\alpha)$, from which it follows that $V(\pm\alpha)=0$. Moreover, as $q$ is increasing on $\mathbb{R}$ and $q(\pm\infty)=\pm\alpha$, then $q$ is an invertible function in $(-\alpha,\alpha)$. Therefore, since $G(t)>0$ for all $t>0$, we must have that $V(t)>0$ for any $t\in(-\alpha,\alpha)$, ending the proof.
\end{proof}


\appendix
\section{Orlicz and Orlicz-Sobolev spaces}\label{s2} 

In what follows we present a brief review about the Orlicz and Orlicz-Sobolev spaces. For the interested reader in this subject, we cite the books \cite{Adams,Rao} and the bibliography therein. To begin with, we say that a function $\Phi:\mathbb{R}\rightarrow [0,+\infty)$ is an \textit{$N$-function} if:
\begin{itemize}
	\item[(a)] $\Phi$ is continuous, convex and even,
	\item[(b)] $\Phi(t)=0$ if and only if $t=0$,
	\item[(c)] $\displaystyle \lim_{t\rightarrow 0}\dfrac{\Phi(t)}{t}=0$ and $\displaystyle \lim_{t\rightarrow +\infty}\dfrac{\Phi(t)}{t}=+\infty$. 
\end{itemize}
Moreover, we say that an $N$-function $\Phi$ verifies the \textit{$\Delta_2$-condition}, in short we write $\Phi\in\Delta_2$, whenever 
$$
\Phi(2t)\leq C\Phi(t)~\text{ for all }~t\geq t_0,\eqno{(\Delta_2)}
$$ \\
holds true for some $C>0$ and $t_0\geq0$. To continue this brief review of Orlicz and Orlicz-Sobolev spaces, let $\Omega$ be an open set of $\mathbb{R}^N$, with $N \geq 1$, and $\Phi$ an $N$-function. In these configurations, we can write the following class of functions 
$$
L^{\Phi}(\Omega)=\left\{u\in L^1_{\text{loc}}(\Omega):~\int_\Omega \Phi\left(\frac{|u|}{\lambda}\right) dx<+\infty~\text{for some}~\lambda>0\right\}.
$$
The space $L^\Phi(\Omega)$ is called the \textit{Orlicz space} associated with $\Phi$ over $\Omega$ and endowed with the norm  
$$
\|u\|_{L^{\Phi}(\Omega)}=\inf\left\{\lambda>0:\int_\Omega \Phi\left(\frac{|u|}{\lambda}\right) dx\leq 1\right\}
$$ 
carries the structure of a Banach space. The norm $u\in L^{\Phi}(\Omega)\mapsto\|u\|_{L^{\Phi}(\Omega)}$ is clearly a Minkowski functional and is called \textit{Luxemburg norm} associated with $\Phi$ over $\Omega$. When $\Phi\in\Delta_2$, 
$$
L^{\Phi}(\Omega)=\left\{u\in L^1_{\text{loc}}(\Omega):~\int_\Omega \Phi(|u|) dx<+\infty\right\},
$$
and moreover, 
$$
u_n\rightarrow u~~ \text{in}~~L^{\Phi}(\Omega)\Leftrightarrow \int_{\Omega}\Phi(|u_n-u|)dx\rightarrow 0.
$$

These spaces have a very rich topological and geometrical structure and they are a very elegant generalization of Lebesgue spaces, because in the case of $\Phi_p(t)=\frac{|t|^p}{p}$ with $p\in(1,+\infty)$, one has 
$$
L^{\Phi_p}(\Omega)=L^{p}(\Omega)\text{ and } \|u\|_{L^{\Phi_p}(\Omega)}=p^{-\frac{1}{p}}\|u\|_{L^{p}(\Omega)},
$$
where $\|\cdot\|_{L^p(\Omega)}$ denotes the usual norm of space $L^p(\Omega)$. The spaces $L^{\Phi}$ are more general than Lebesgue's $L^p$ spaces, and may have peculiar properties that do not occur in ordinary Lebesgue's spaces. However, in some cases there are relationships between them, such as continuous embedding
\begin{equation}\label{emb}
	L^\Phi(\Omega)\hookrightarrow L^1(\Omega),	
\end{equation} 
whenever Lebesgue measure on $\mathbb{R}^N$ of $\Omega$ is finite (in short, $|\Omega|<+\infty$). Other examples of $N$-functions that satisfy $(\Delta_2)$ with $t_0=0$ and generate interesting Orlicz spaces are
\begin{itemize}
	\item[(1)] $\Phi_1(t)=\dfrac{|t|^p}{p}+\dfrac{|t|^q}{q}$ with $p<q$ and $p,q\in(1,+\infty)$,
	\item[(2)] $\Phi_2(t)=(1+t^2)^{\gamma}-1$ with $\gamma>1$,
	\item[(3)] $\Phi_3(t)=|t|^p\ln(1+|t|)$ for $p\in(1,+\infty)$,
	\item[(4)] $\Phi_4(t)=\displaystyle\int_{0}^{t}s^{1-\gamma}(\sinh^{-1}s)^{\beta}ds$ with $0\leq\gamma< 1$ and $\beta>0$.
\end{itemize}

For a more general prototype, we would like to point out that the function $\Phi$ of the type \eqref{*} satisfying $(\phi_1)$ and $(\phi_2)$ is an $N$-function. Moreover, examples (1) to (4) listed above are templates of the form \eqref{*} checking $(\phi_1)$-$(\phi_2)$.

We are now interested in emphasizing that given an $N$-function, we can obtain another $N$-function that plays a fundamental role in the Orlicz's Spaces Theory. Indeed, given an $N$-function $\Phi$, the complementary function $\tilde{\Phi}$ associated with $\Phi$ is defined by Legendre’s transformation
$$
\tilde{\Phi}(s)=\max_{t\geq 0}\left\{st-\Phi(t)\right\}\text{ for } s\geq 0,
$$ 
which is also an $N$-function. Moreover, the functions $\Phi$ and $\tilde{\Phi}$ are complementary each other and satisfy the following inequality 
$$
st\leq \Phi(t)+\tilde{\Phi}(s),~~~\forall s,t\geq 0. \quad (\mbox{Young type inequality})
$$ 
A classic example of such complementary functions are
$$
\Phi(t)=\frac{|t|^p}{p}\text{ and }\tilde{\Phi}(t)=\frac{|t|^q}{q},\text{ where }1<p<q<+\infty\text{ and }\frac{1}{p}+\frac{1}{q}=1.
$$

We will see in the next result that the $\Delta_2$-condition plays an essential role in the treatment of Orlicz-Sobolev spaces.

\begin{lemma}\label{A1}
	The space $L^{\Phi}(\Omega)$ is reflexive if, and only if, $\Phi$ and $\tilde{\Phi}$ satisfy the $\Delta_2$-condition. 
\end{lemma}
\begin{proof}
	See for instance \cite[Ch. IV, Theorem 10]{Rao}.
\end{proof}

A generalization of Sobolev spaces $W^{1,p}$ is the \textit{Orlicz-Sobolev space}, which is defined as the following Banach space associated with $\Phi$ over $\Omega$
$$
W^{1,\Phi}(\Omega)=\left\{u\in L^{\Phi}(\Omega):~\frac{\partial u}{\partial x_i}=u_{x_i}\in L^{\Phi}(\Omega),~i=1,...,N\right\},
$$ 
equipped with the norm
$$
\|u\|_{W^{1,\Phi}(\Omega)}=\|\nabla u\|_{L^{\Phi}(\Omega)}+\|u\|_{L^{\Phi}(\Omega)},
$$ 
where $\nabla u=(u_{x_1},...,u_{x_N})$.

Next, we list some results about $N$-functions of the form \eqref{*} satisfying the conditions $(\phi_1)$ and $(\phi_2)$, which will be used frequently in this work. 

\begin{lemma}\label{A2}
	Let $\Phi$ be an $N$-function of the type \eqref{*} satisfying $(\phi_1)$-$(\phi_2)$ and set the functions  
	$$
	\xi_0(t)=\min\{t^{l},t^{m}\} \quad \mbox{and} \quad \xi_1(t)=\max\{t^{l},t^{m}\} \quad \mbox{for} \quad t\geq 0.
	$$ 
	Then $\Phi$ satisfies 
	$$
	\xi_0(t)\Phi(s)\leq\Phi(st)\leq\xi_1(t)\Phi(s),~~\forall s,t\geq0.
	$$
	In particular, $\Phi\in\Delta_2$.
\end{lemma}
\begin{proof} Notice that  by $(\phi_2)$, one has  
	$$
	l\phi(t)t\leq\left(\phi(t)t^2\right)'\leq m\phi(t)t,~~\forall t>0,
	$$
	which yields 
	\begin{equation}\label{0}
		l\leq \dfrac{\phi(t)t^{2}}{\Phi(t)}\leq m,~~\forall t>0.
	\end{equation}
	Having this in hand, we can use \cite[Lemma 2.1]{Fukagai} to get the desired result. 
\end{proof}

\begin{lemma}\label{A3}
	If $\Phi$ is an $N$-function of the type \eqref{*} satisfying $(\phi_1)$-$(\phi_2)$, then $\tilde{\Phi}\in\Delta_2$. 
\end{lemma}
\begin{proof}
	See \cite[Lemma 2.7]{Fukagai} for the proof.
\end{proof}

\begin{lemma}\label{A4}
	If $\Phi$ is an $N$-function of the form \eqref{*} satisfying $(\phi_1)$-$(\phi_2)$, then the spaces $L^{\Phi}(\Omega)$ and $W^{1,\Phi}(\Omega)$ are reflexive.
\end{lemma}
\begin{proof}
	The proof follows naturally from Lemmas \ref{A1}, \ref{A2} and \ref{A3}.
\end{proof}

\begin{lemma}\label{A6}
	If $\Phi$ is an $N$-function of the type \eqref{*} satisfying $(\phi_1)$-$(\phi_2)$, then 
	$$
	\tilde{\Phi}(\phi(t)t)\leq\Phi(2t)\text{ for all }t\geq 0.
	$$
\end{lemma}
\begin{proof}
	See for instance \cite[Lemma A.2]{Fukagai}.
\end{proof}

Finally, we would like to list here some elementary inequalities involving $N$-functions of the form \eqref{*}, whose the proofs will be omit because they are very simple. 

\begin{lemma}\label{A5}
	Let $\Phi$ be an $N$-function of the type \eqref{*} satisfying $(\phi_1)$-$(\phi_2)$. Then the following inequalities hold
	\begin{itemize}
		\item[(a)] $\Phi(|s+r|)\leq 2^m\left(\Phi(|s|)+\Phi(|r|)\right)$ for all $s,r\in\mathbb{R}$.
		\item[(b)] $\left(\phi(|s|)s-\phi(|r|)r\right)(s-r)>0$ for all $s,r\in\mathbb{R}$ with $s\neq r$.	
		\item[(c)] $\phi(|z|)z\cdot(w-z)\leq\Phi(|w|)-\Phi(|z|)$ for all $w,z\in\mathbb{R}^N$ with $z\neq0$.
		
	\end{itemize}
\end{lemma}


\end{document}